\documentclass[11pt]{article}
\usepackage[dvips,letterpaper,margin=1.0in]{geometry}

\usepackage[dvipsnames]{xcolor}

\usepackage{sty}

\usepackage{subcaption}
\usepackage{tikz}
\usepackage{graphicx}
\usepackage{tikz-cd}

\usepackage{hyperref}

\title{Computational Complexity of Statistics:\\New Insights from Low-Degree Polynomials}

\usepackage{authblk}
\author{Alexander S.\ Wein\thanks{Email: \textit{aswein@ucdavis.edu}. Supported by a Sloan Research Fellowship and NSF CAREER Award CCF-2338091.}}
\affil{Department of Mathematics, UC Davis}
\date{}

\begin{document}

\maketitle

\begin{abstract}
This is a survey on the use of low-degree polynomials to predict and explain the apparent statistical-computational tradeoffs in a variety of average-case computational problems. In a nutshell, this framework measures the complexity of a statistical task by the minimum degree that a polynomial function must have in order to solve it. The main goals of this survey are to (1)~describe the types of problems where the low-degree framework can be applied, encompassing questions of detection (hypothesis testing), recovery (estimation), and more; (2)~discuss some philosophical questions surrounding the interpretation of low-degree lower bounds, and notably the extent to which they should be treated as evidence for inherent computational hardness; (3)~explore the known connections between low-degree polynomials and other related approaches such as the sum-of-squares hierarchy and statistical query model; and (4)~give an overview of the mathematical tools used to prove low-degree lower bounds. A list of open problems is also included.
\end{abstract}

\newpage
\tableofcontents

\newpage
\section{Overview}

\subsection{What Is This Topic?}

Imagine trying to find a hidden $k$-vertex clique (fully connected subgraph) within an otherwise random $n$-vertex graph (network). While it is possible to find a hidden clique of size $k \approx \log n$ by brute-force search, all known ``fast'' (polynomial-time) algorithms only work if the clique is much larger: $k \approx \sqrt{n}$. Is this an inherent limitation of fast algorithms or should we continue looking for a better one? Similar questions of computational complexity arise in many other statistical settings, such as community detection, clustering, and sparse PCA. While we lack the tools to prove definitively that fast algorithms require $k \gtrsim \sqrt{n}$, this survey describes one sense in which we \emph{can} prove this threshold is fundamental: all algorithms based on \emph{low-degree polynomials} --- for instance, counting triangles in the graph would be a degree-3 polynomial --- provably fail (in an appropriate sense) when $k \ll \sqrt{n}$. Furthermore, these low-degree algorithms tend to capture the best tools in our algorithmic toolkit for problems of this style, so finding a fast algorithm for $k \gg \sqrt{n}$ would seem to require a major breakthrough or may simply be impossible. This provides a lens for predicting and explaining the limitations of fast algorithms across many different settings.

\subsection{Why This Survey?}

In 2019, I wrote a survey~\cite{ld-notes} on this topic together with Dmitriy (Tim) Kunisky and Afonso Bandeira. In the 6 years since then, our community has refined some of the basic notions in this area, as well as the surrounding philosophy. For instance, my opinion about the ``right'' definition of ``success'' for low-degree algorithms has evolved somewhat over time (see the discussion in Section~\ref{sec:comm-sep}). Furthermore, low-degree lower bounds now exist not only for hypothesis testing questions, but also for other types of tasks such as estimation and optimization. Finally, new discoveries have surfaced in the ongoing quest to understand the relation between predictions based on low-degree polynomials and predictions based on other frameworks for average-case complexity. This survey aims to give an updated account of all these new developments.

When writing a paper on this topic, I always find myself attempting to write a section that justifies why the low-degree polynomial framework yields reliable conjectures about inherent computational hardness, while also remaining transparent about many caveats and imploring readers not to be overconfident about these conjectures. These discussions are nuanced and it is tedious to repeat them in every paper. This survey (particularly Section~\ref{sec:interp}) is my attempt to address these points as best I can, so that future papers can simply refer the reader here for such discussions.

In contrast to our previous survey~\cite{ld-notes}, this one is mostly non-technical: I won't give the full details of any proofs here. Instead, this is intended as a ``roadmap'' to the existing literature. For some of the more subjective points, I will sometimes inject my own opinion and will use ``I'' rather than the academic ``we'' to designate this. At the same time, I will attempt to represent other opinions and philosophies that I feel are present in our community. Please feel free to contact me with any corrections or counterpoints to the arguments I've made here.

\subsection{Acknowledgments}

This survey was written with significant input from Tim Kunisky and Tselil Schramm, who both gave me detailed feedback on an earlier draft. I am also grateful for comments from Guy Bresler, Ilias Diakonikolas, David Gamarnik, Cheng Mao, Ankur Moitra, Aaron Potechin, and Ilias Zadik, particularly on the connections to other frameworks in Section~\ref{sec:relations}. Many others have helped shape my ideas about the topics in this survey, including Sam Hopkins.

\subsection{Organization}
\label{sec:org}

Part~I (Sections~\ref{sec:high-dim}--\ref{sec:low-deg}) is intended as a self-contained introduction to this area. Section~\ref{sec:high-dim} sets the stage by describing the style of problems where the low-degree polynomial framework may be applicable. Section~\ref{sec:low-deg} explains what it means to apply the low-degree framework, and what exactly we aim to prove in order to characterize the complexity of statistical tasks using polynomial degree.

Part~II (Sections~\ref{sec:commentary}--\ref{sec:open}) is intended as a deeper dive into the philosophy surrounding the low-degree framework. Each of the six sections in Part~II is mostly self-contained, and those sections can be read in any order. Section~\ref{sec:commentary} offers additional commentary on the definitions from Part~I, as there are a number of frequently asked questions to address here. Section~\ref{sec:tasks} describes specific statistical objectives where the low-degree framework can be applied --- detection, recovery, optimization, and refutation --- as well as the relations among them. Section~\ref{sec:interp} contains a more conjectural discussion of what ``degree complexity'' means for the more traditional \emph{time complexity}, including the ``low-degree conjecture.'' Section~\ref{sec:relations} attempts to summarize the known relations among different frameworks for average-case complexity, with low-degree polynomials playing a central role in this web of connections. Section~\ref{sec:techniques} surveys the mathematical tools that exist for proving the types of results described in Section~\ref{sec:low-deg}. Finally, Section~\ref{sec:open} lists some open problems, some of which are also mentioned throughout the text.

\subsection{Notation and Conventions}

Some basic mathematical notation: $[n]$ is shorthand for $\{1,2,\ldots,n\}$ and $\|\cdot\|$, when applied to a vector, denotes the Euclidean $\ell^2$-norm. The \emph{degree} of a multivariate polynomial is the highest sum of exponents in any term, so for instance, the degree of $f(x,y,z) = 2 x^3 y^5 + xyz^2$ is $8$. The indicator $\One_A$ takes value 1 if the predicate $A$ is true, and 0 if $A$ is false.

Asymptotic notation: we will often consider an asymptotic regime where some parameter called $n$ tends to infinity. We imagine fixing a scaling for the other parameters, so that all parameters can be treated as functions of $n$. Some parameters may be designated as ``constants'' that do not depend on $n$. We use the standard meaning of ``big-O''/``little-o'' notation such as $O(\cdot), \Omega(\cdot), \Theta(\cdot), o(\cdot), \omega(\cdot)$, always pertaining to the limit $n \to \infty$ and potentially hiding constants that do not depend on $n$. Also, $\poly(n)$ stands for $n^{O(1)}$ and $\polylog(n)$ stands for $(\log n)^{O(1)}$. We write $\log^m n$ as shorthand for $(\log n)^m$. The notation $\tilde{O}(\cdot), \tilde{\Omega}(\cdot)$ hides a $\polylog(n)$ factor. The term ``high probability'' means probability $1-o(1)$.

The notation $f(n) \ll g(n)$ is used more informally, typically in a scenario where we seek to identify the ``right'' power of $n$ that some parameter should scale as. So $k \ll \sqrt{n}$ may hide a $\polylog(n)$ factor, or possibly even an $n^{o(1)}$ factor.

The terms ``lower bound'' and ``upper bound'' will often be used in the computer science sense, that is, lower bounds are negative results that show failure for some class of algorithms, while upper bounds are positive results that show success for some algorithm in the class.

References to specific theorems or sections of other works generally refer to the numbering used in the current arXiv version (or whatever freely available version was easiest to find).  Citations are generally listed in chronological order based on the date of first appearance (e.g., on arXiv).

\newpage
\phantomsection
\begin{center}
\huge{\bf Part I: The Basics}
\end{center}
\addcontentsline{toc}{section}{Part I: The Basics}

Part I is intended as a relatively short introduction to the low-degree polynomial framework, accessible to those who may not be familiar with this area.

\section{High-Dimensional Statistics and Computational Barriers}
\label{sec:high-dim}

In the age of ``big data'' we might seek to design an algorithm for extracting useful information from some large noisy dataset. Ideally, this algorithm should reach the fundamental limits of how much can be learned from a limited quantity or quality of data. The desire to understand these limits motivates the mathematical models discussed below.

\subsection{A Few Motivating Examples}
\label{sec:mot-ex}

The framework we will discuss in this survey is useful for understanding a particular style of problems, which can broadly be described as \emph{high-dimensional statistical inference} (\emph{high-dim stats} for short). Before attempting to explain what this means in general, it is instructive to see a few canonical examples.

\begin{itemize}

\item {\bf Planted Clique} (e.g.~\cite{alon-clique,pcal}): Suppose a random graph (network) on $n$ vertices (nodes) is generated according to the Erd\H{o}s--R\'{e}nyi distribution $G(n,1/2)$, that is, each of the $\binom{n}{2}$ potential edges is included independently with probability $1/2$. Then suppose $k$ out of the $n$ vertices are chosen at random and a \emph{clique} (fully-connected subgraph) is included (``planted'') on those vertices, that is, every edge between two of those $k$ vertices that is not already present, is added to the graph. Finally, we (the statistician) observe the resulting graph and our goal is to determine on which $k$ vertices the clique was planted.

\item {\bf Sparse PCA\footnote{The term \emph{PCA (principal component analysis)} sometimes refers to a technique for data analysis, but here we use it to refer to a statistical model where such techniques might be employed. The distinction between models and methods is an important one! The model here is a particular variant of sparse PCA in the \emph{spiked Wishart} model.}} (e.g.~\cite{cov-thresholding,BB-opt}): A hidden $k$-sparse vector $x \in \RR^n$ is drawn from the \emph{sparse Rademacher distribution}, that is, $k$ randomly selected entries of $x$ take values $\pm 1/\sqrt{k}$ with uniformly random signs, and the rest are $0$. This normalization ensures $\|x\| = 1$. Then, for some parameter $\beta \ge 0$, we observe $N$ independent samples $y_1,\ldots,y_N$ where each $y \in \RR^n$ is drawn independently from the multivariate Gaussian distribution $\cN(0,I_n+\beta xx^\top)$ where $I_n$ denotes the $n \times n$ identity matrix. Our goal is to recover $x$, which represents a sparse direction of high variance in the data. However, it is impossible to differentiate $x$ from $-x$, so we actually aim to recover $x$ up to sign, or equivalently, to recover $xx^\top$.

\end{itemize}

\noindent With varying degrees of realism, these serve as simplified models for large-scale inference problems on large data sets.\footnote{Sparse PCA seems to qualitatively capture some real-world challenges, whereas planted clique gained popularity largely due to its connections with other problems; see Section~\ref{sec:reductions} on average-case reductions.} As such, we will typically be interested in an asymptotic regime where $n$ --- which will always be some notion of ``dimension'' or ``problem size'' --- tends to infinity. This means we are really considering a sequence of inference problems, indexed by $n$, where other parameters such as $k,N,\beta$ might scale with $n$ in some prescribed way, or might be designated as fixed ``constants'' which do not depend on $n$. Our objective, as the statistician, can take various precise forms that will be covered later, but a natural one to have in mind for the problems above is \emph{exact recovery} with \emph{high probability}, that is, we aim to exactly determine the set of clique vertices or the matrix $xx^\top$, with success probability that provably tends to $1$ as $n \to \infty$. This probability is over the random instance (i.e., the random graph or the samples $y_i$) as well as any internal randomness used by our inference procedure.

\subsection{Statistical-Computational Gaps}

It is desirable to have a procedure that is \emph{statistically optimal}, meaning it succeeds under the widest possible conditions on the parameters such as $k,N,\beta$. Additionally, it is desirable to have a procedure that is \emph{computationally efficient}, meaning the statistician uses an algorithm of practical runtime, which we will usually take to mean a \emph{polynomial-time} (\emph{poly-time} for short) algorithm. However, there appears to be an inherent tension between statistical and computational efficiency, as the following examples show.

\begin{itemize}
\item {\bf Planted Clique}: The size (number of vertices) of the largest naturally-occurring clique in $G(n,1/2)$ concentrates near $2 \log_2 n$~\cite[Ch~11]{boll-book}, and it becomes possible to identify the planted clique (with high probability) as soon as $k \ge (2+\epsilon) \log_2 n$ for an arbitrary constant $\epsilon > 0$. Conversely, when $k \le (2-\epsilon) \log_2 n$, it is provably impossible to identify the planted clique, no matter what procedure is used.\footnote{Actually, $2 \log_2 n$ is known to be the information-theoretic threshold for the related \emph{detection} problem~\cite{AV-clique}, and I am not aware of the same being shown for recovery, although this is expected to be true. I later hope to make the point that detection and recovery should not be conflated, so I am already setting a bad example here!} We therefore call $k = 2 \log_2 n$ the \emph{statistical} or \emph{information-theoretic} threshold. However, the algorithm reaching this threshold involves an exhaustive (``brute force'') search over subsets of vertices of size $O(\log n)$, and has runtime $n^{O(\log n)}$. The best known \emph{poly-time} algorithms require $k \ge \epsilon \sqrt{n}$, again for an arbitrary constant $\epsilon > 0$, where the runtime has the form $n^{O(\log 1/\eps)}$~\cite{alon-clique}. Despite much effort over the last decades, no better algorithm has been found, and so we refer to $k \approx \sqrt{n}$ as the \emph{computational} or \emph{algorithmic} threshold. The ``phase diagram'' is depicted in Figure~\ref{fig:clique} with three regimes for $k$: ``impossible,'' ``hard,'' and ``easy.''\footnote{Here and throughout we use ``easy'' in the formal sense: there exists a poly-time algorithm. It is not meant as a judgment of how sophisticated (or not) the algorithm is!}

\item {\bf Sparse PCA}: Consider the scaling $N = \Theta(n)$, $k = \Theta(n^a)$, and $\beta = \Theta(n^b)$ for constants $a \in (0,1)$ and $b \in \RR$. There is extensive literature on this model and the results are summarized in Figure~1 of~\cite{BB-opt}. Exact recovery is information-theoretically impossible when $b < \frac{1}{2}(a-1)$ but possible when $b > \frac{1}{2}(a-1)$ by exhaustive search. Poly-time algorithms are known when $b > 0$ or $b > a-\frac{1}{2}$, but no such algorithms are known in the remaining ``hard'' region $\frac{1}{2}(a-1) < b < \min\{0,a-\frac{1}{2}\}$. Figure~\ref{fig:sparse-pca} shows the resulting 2-dimensional phase diagram. Notably,~\cite{BR-reduction} was first to bring to light the inherent statistical-computational tradeoff.
\end{itemize}

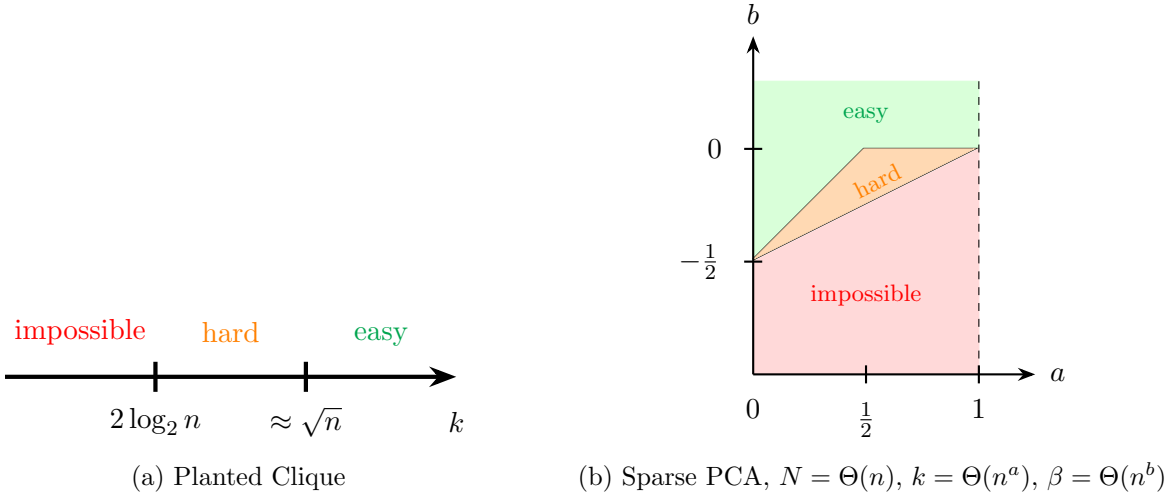
\begin{figure}
    \centering
    \begin{subfigure}[t]{0.4\textwidth}
        \centering
        \begin{tikzpicture}[xscale=2]
            \def \d{.2}; % tick mark height
            \def \y{.6}; % text to line distance
            % axis
            \draw[ultra thick, -Stealth] (0,0) -- (3,0);
            % tick marks
            \draw[ultra thick] (1,-1*\d) -- (1,\d);
            \draw[ultra thick] (2,-1*\d) -- (2,\d);
            % bottom labels
            \node at (1,-1*\y) {$2 \log_2 n$};
            \node at (2,-1*\y) {$\approx \sqrt{n}$};
            \node at (3,-1*\y) {$k$};
            % above labels
            \node at (.5, \y) {\color{red} impossible};
            \node at (1.5, \y) {\color{orange} hard};
            \node at (2.5, .90*\y) {{\color{Green} easy}};
        \end{tikzpicture}
        \caption{Planted Clique}
        \label{fig:clique}
    \end{subfigure}%
    ~ 
    \begin{subfigure}[t]{0.6\textwidth}
        \centering
        \begin{tikzpicture}[scale=3]
            \def \d{.04}; % tick mark height
            \def \y{.1}; % text to line distance
            %%%%
            % easy region
            \filldraw[green!15!white] (0,-.5) -- (.5,0) -- (1,0) -- (1,.3) -- (0, .3) -- (0, -.5);
            \node at (.5, .15) {\color{Green} \footnotesize{easy}};
            % impossible region
            \filldraw[red!15!white] (0,-1) -- (1,-1) -- (1,0) -- (0,-.5) -- (0,-1);
            \node at (.5, -.65) {\color{red} \footnotesize{impossible}};
            % hard region
            \begin{scope}[scale=.98]
                \draw[thick] (0,-.5) -- (1,0) -- (.5,0) -- (0,-.5);
                \filldraw[orange!30!white] (0,-.5) -- (1,0) -- (.5,0) -- (0,-.5);
            \end{scope}
            \node at (.55, .-.14) {
                \rotatebox{26.6}{\color{orange} \footnotesize{hard}}
                };
            %%%%
            % axes
            \draw[thick, -Stealth] (0,-1) -- (1.25, -1);
            \draw[thick, -Stealth] (0,-1) -- (0,.5);
            % axis labels
            \node at (0,.5+\y) {$b$};
            \node at (1.25+\y,-1) {$a$};
            % tick marks
            \draw[thick] (-1*\d,-.5) -- (\d,-.5);
            \node at (-1*\d-2*\y,-.5) {$-\frac{1}{2}$};
            \draw[thick] (-1*\d,0) -- (\d,0);
            \node at (-1*\d-1.3*\y,0) {$0$};
            \node at (0, -1 - 1.5*\y) {$0$};
            \draw[thick] (.5,-1*\d-1) -- (.5,\d-1);
            \node at (.5, -1 - 2*\y) {$\frac{1}{2}$};
            \draw[thick] (1,-1*\d-1) -- (1,\d-1);
            \node at (1, -1 - 1.5*\y) {$1$};
            %%%
            % dotted line on right side
            \draw[dashed] (1,-1) -- (1,.3);
        \end{tikzpicture}
        \caption{Sparse PCA, $N = \Theta(n)$, $k = \Theta(n^a)$, $\beta = \Theta(n^b)$}
        \label{fig:sparse-pca}
    \end{subfigure}
    \caption{Phase diagrams for the two models defined in Section~\ref{sec:mot-ex}.}
\end{figure}

\noindent Some parts of these phase diagrams can be confirmed rigorously: in the ``easy'' regime we can demonstrate an algorithm and prove that it works, and in the ``impossible'' regime we can apply tools from classical statistics to prove that no successful estimator exists (regardless of runtime). However, both examples above exhibit a so-called \emph{statistical-computational gap} (\emph{stat-comp gap} for short), meaning there appears to be a ``possible-but-hard'' regime. While we can establish ``possible'' by analyzing a brute-force estimator, the ``hard'' aspect is more elusive and will be our main focus in this survey: is there fundamentally no efficient (poly-time) algorithm here, or have we just not discovered the right algorithm yet? We want to find better and better algorithms, but we do not want to keep searching indefinitely for a better algorithm when it may be futile, so we'd like a way to prove the hardness is fundamental, in some appropriate sense. The urgency of resolving this is amplified by the fact that stat-comp gaps occur in many different settings, and we mention a few more of the most ``classical'' examples below. See Section~\ref{sec:list} for an even longer list of examples.

\begin{itemize}
    \item {\bf Spiked Matrix Models} (e.g.~\cite{miolane-survey}): This refers to a general class of models where a random matrix (``noise'') is perturbed by a low-rank ``signal.'' The sparse PCA problem defined in Section~\ref{sec:mot-ex} is a special case of the \emph{spiked Wishart} model, where more generally the ``spike'' $x$ can be drawn from any distribution (the ``prior'') over unit vectors. There is also a \emph{negatively spiked} variant where $\beta \in [-1,0)$; see~\cite{sk-cert}. In the related \emph{spiked Wigner} model we observe an $n \times n$ matrix $Y = \lambda xx^\top + Z$ where $\lambda \ge 0$ is the \emph{SNR (signal-to-noise ratio)}, the spike $x$ is again drawn from some prior over unit vectors, and $Z$ is an i.i.d.\ (or symmetric) Gaussian matrix. The goal is to recover (perhaps approximately) $xx^\top$, given $Y$.

    \item {\bf Tensor PCA} (e.g.~\cite{RM-tensor-pca}): This is a generalization of the above spiked Wigner model to tensors, meaning multi-dimensional arrays. In the third-order case, we observe an $n \times n \times n$ tensor $Y = \lambda x^{\otimes 3} + Z$ where $\lambda \ge 0$ is the SNR, $Z$ is an i.i.d.\ Gaussian tensor, and the rank-1 tensor $x^{\otimes 3}$ is defined as $(x^{\otimes 3})_{ijk} = x_i x_j x_k$. The goal is to recover $x$, given $Y$.

    \item {\bf Stochastic Block Model} (e.g.~\cite{moore-sbm-survey,abbe-survey}): This is a popular model for community detection in random graphs. An $n$-vertex graph is generated by first partitioning the vertices into $q$ ``communities'' and then connecting each pair of vertices with probability $p$ for within-community pairs or probability $q$ for across-community pairs. The goal is to infer the hidden partition, given the graph.
    
\end{itemize}

\noindent Some survey articles about stat-comp gaps from various perspectives include~\cite{ZK-phys-survey,notes-phys,WX-survey,sos-survey,ld-notes,ogp-survey,GMZ-phys-survey,turing-survey}. We will discuss many of these perspectives later in Section~\ref{sec:relations}.

\subsection{Characteristics of High-Dim Stat Problems}

Throughout this survey, I will often refer to some informal notion of ``natural'' high-dim stat problems, and sometimes make general claims that are intended to apply only to problems of this style. The examples above aim to clarify the meaning of this. Section~\ref{sec:conj} on the low-degree conjecture will discuss the prospect of formally defining the class of problems we are interested in. For now, let's recap some of the common features of high-dim stat problems.

First, the meaning of ``high-dimensional'' is that the number of unknowns scales up with the amount of data, as reflected by the scaling $N = \Theta(n)$ in the sparse PCA example above. In contrast, the ``classical'' regime in statistics would be many samples ($N \to \infty$) coming from a fixed distribution (i.e., $n,k,\beta$ fixed), which tends not to have stat-comp gaps. In the classical regime, one expects to approach near-perfect inference as $N$ increases, but in the high-dimensional regime there is no guarantee of this, depending on how the various parameters scale.

Throughout, our focus will be on problems that are fully Bayesian, meaning the statistician has full knowledge of the distribution that generated the data, including all its parameters. For instance, in the planted clique problem, the location of the clique is drawn from the uniform prior over all size-$k$ subsets of $[n]$, and $k$ is known. We should not imagine the clique location as fixed, or else there exists a trivial algorithm with the correct answer hard-coded! For ``upper bounds'' (algorithm design and analysis), it may be desirable to relax the Bayesian assumption and consider inputs with components that are arbitrary and unknown (as in \emph{minimax} theory~\cite{minimax}), or even adversarial (as in \emph{semirandom models}~\cite{semirandom} or \emph{contamination models}~\cite{robust-est,robust-stat-book}). However, our main focus will be on lower bounds, where it commonly suffices to show that even some ``easier'' Bayesian problem is hard.

Another common feature of the problems above is the presence of a ``planted signal'' that we aim to recover, although later in Section~\ref{sec:tasks} we will also discuss ``non-planted'' optimization and refutation problems where the input is pure random ``noise.'' Furthermore, the distribution over inputs is highly symmetric. For instance, the planted clique distribution (including the random choice of clique vertices) is invariant under permutation of the vertices.

Finally, why do we care to study this class of problems, aside from an inherent mathematical interest? To varying degrees, some of these models may capture real-world phenomena, at least qualitatively. This may be especially true in settings like group testing (see~\cite{grp-testing}) or cryptography (see~\cite{crypto-graph,crypto-hyperloops,crypto-random-subgraph}) where we have control over the distributions involved. Even the planted clique model, which has no strong practical motivation, serves an important role as a testbed for building mathematical tools that we hope are widely applicable. One end goal is to develop a user-friendly theory so that whenever a new problem comes along, we can systematically find the \emph{optimal} algorithm(s) --- including a matching hardness result --- and understand the inherent tradeoffs between statistical resources (quantity and quality of data) and computational resources (compute power). As this theory grows stronger, it can be applied to increasingly complex models inspired by applications. Finally, if we've proved hardness for an idealized version of our real-world problem, this likely means we should not expect to solve the (presumably more difficult) real thing, and should instead focus on making the problem easier (e.g., by collecting more data).

\subsection{Average-Case Complexity Theory?}

For high-dim stat problems, we would like a way to prove convincing hardness results, but we cannot hope to prove outright that no poly-time algorithm exists, as this would resolve the notorious P vs NP question. The notion of NP-hardness from classical complexity theory has been widely successful at classifying problems as ``easy'' or (conjecturally) ``hard'' for \emph{worst-case} computational tasks, where ``worst-case'' means we demand an efficient algorithm that works for \emph{every} possible instance (input). In contrast, our setting is that of \emph{average-case} complexity, where a successful algorithm need only work for ``most'' random inputs drawn from a specific distribution. To illustrate the difference, the maximum clique problem is NP-hard but this doesn't stop us from efficiently solving planted clique for \emph{some} values of $k$. We currently do not have the tools to prove hardness of high-dim stat problems conditional on P $\ne$ NP. Instead, a number of different approaches exist for predicting and explaining stat-comp gaps (unfortunately, none quite as satisfying as the theory of NP-hardness; see Section~\ref{sec:relations}). In the next section, we will describe one of these theories.

\section{Low-Degree Algorithms}
\label{sec:low-deg}

What do the best known poly-time algorithms look like for the types of problems described in the previous section? Taking the planted clique problem for example, the algorithm of~\cite{alon-clique} for recovering a clique of size $k=O(\sqrt{n})$ is a \emph{spectral method}, meaning it uses the leading eigenvector(s) of some matrix associated to the graph (in this case, the adjacency matrix). If we ignore logarithmic factors and are satisfied with recovering a clique of size $k = O(\sqrt{n \log n})$, an even simpler approach will work: simply pick out the vertices of largest degree (number of incident edges). Related to the task of recovering a planted clique is the simpler task of \emph{detecting} a planted clique, where the goal is to decide whether a given graph was drawn from the planted clique distribution (for some given $k$) or the ``null'' distribution $G(n,1/2)$. As long as $k = \omega(\sqrt{n})$, an extremely simple algorithm solves detection with high probability: threshold the total edge count. The analysis is a routine calculation of mean and variance, combined with Chebyshev's inequality.

Variations of these basic algorithmic ideas have proved to be widely useful for obtaining the best known algorithmic guarantees for a variety of high-dim stat problems. In general, spectral methods may use a more sophisticated choice of matrix (with some examples mentioned in Section~\ref{sec:alg-captured}). In some settings, rather than simply counting total edges or edges incident to a particular vertex, it may be fruitful to count instances of some other subgraph, like triangles or 4-cycles. We will see that all the algorithms we've mentioned here can be put under a common umbrella: they are all based on \emph{low-degree polynomials}.

\subsection{Setup}

We now introduce low-degree polynomials as a class of potential approaches for high-dimensional detection and recovery problems. Consider for example the planted clique detection problem, with the input graph represented using the variables $Y = (Y_{ij})_{1 \le i < j \le n}$ where $Y_{ij} \in \{0,1\}$ encodes the presence (1) or absence (0) of edge $(i,j)$. We will consider a multivariate polynomial $f \in \RR[Y]$, which describes a function $f: \{0,1\}^{\binom{n}{2}} \to \RR$. One could imagine solving the detection problem by thresholding the output of some polynomial, for example:
\begin{itemize}
    \item edge count: $f(Y) = \sum_{1 \le i < j \le n} Y_{ij}$,
    \item triangle count: $f(Y) = \sum_{1 \le i < j < k \le n} Y_{ij} Y_{ik} Y_{jk}$.
\end{itemize}
We will think of polynomials as algorithms, where the polynomial degree serves as a measure of the algorithm's complexity. For instance, the edge count has degree 1, and the triangle count has degree 3. More generally, counting copies of some $d$-edge subgraph has degree $d$. The degree bound $D = D_n$ will often grow with the problem size $n$ in some prescribed way. Historically, this viewpoint arose from the line of work~\cite{pcal,sos-detect,HS-bayesian,hopkins-thesis}. To be clear, this is distinct from the concept of ``polynomial time,'' which means the runtime scales as a polynomial function of $n$; here we are considering algorithms that literally compute polynomial functions of the input variables.

For recovery (rather than detection), we are often interested in recovering a hidden vector, and by symmetry it often suffices to recover a single entry. For instance, in the planted clique recovery problem we can imagine estimating $x := \One_{1 \in S}$, the $\{0,1\}$-valued indicator that vertex $1$ belongs to the set of clique vertices $S$. One could imagine estimating this based on the output of some polynomial, a simple example being the degree\footnote{This is the graph-theoretic degree, not to be confused with polynomial degree!} of vertex $1$: $f(Y) = \sum_{2 \le i \le n} Y_{1,i}$, which is a degree-1 polynomial.

\subsection{Heuristic Degree-Runtime Correspondence}
\label{sec:d-r-corr}

Our viewpoint will be to treat polynomial degree as an intrinsic measure of an algorithm's complexity. This ``degree complexity'' gives a more tractable alternative to the more traditional ``time complexity.'' For intuition on how these two relate, it will be helpful to have in mind the following heuristic correspondence between polynomial degree and runtime, inspired by~\cite{sos-detect,HS-bayesian,hopkins-thesis}. This is meant as an informal guiding principle that should not be taken too literally at this point.

\begin{hyp}[Informal]
\label{hyp:poly-time}
For ``natural'' high-dim stat problems,
\[ \emph{degree-$O(1)$ polynomials} \;\subseteq\; \emph{polynomial-time algorithms} \;\subseteq\; \emph{degree-$O(\log n)$ polynomials} \]
where ``$A \subseteq B$'' means class $B$ is at least as powerful as class $A$.
\end{hyp}

\noindent As usual, $n$ is some natural notion of problem size, which will always be polynomially related to the number of input variables. The first ``$\subseteq$'' above is justified because a degree-$O(1)$ polynomial has $n^{O(1)}$ terms, so a naive term-by-term computation allows it to be evaluated in polynomial time (assuming the coefficients are easy to compute). The second ``$\subseteq$'' is a much more nuanced claim, since there are certainly some high-degree polynomials that can be evaluated quickly (like the determinant of a matrix), and we are essentially asserting that such polynomials are not useful\footnote{Or rather, not any more useful than their lower degree counterparts.} for ``natural'' high-dim stat problems. The choice of logarithmic degree here originates from~\cite[Conjecture~2.2.4]{hopkins-thesis}, where the slightly higher degree $(\log n)^{1+\Omega(1)}$ is used, to be safe. We will discuss the justification for this in Section~\ref{sec:interp}, but for now, suffice it to say that various powerful algorithms can be captured by degree-$O(\log n)$ polynomials. Notably, we will see in Remark~\ref{rem:log-deg} that certain spectral methods can be implemented by a degree-$O(\log n)$ polynomial via power iteration. For this reason, ``low'' degree will often mean degree $O(\log n)$.

For super-polynomial runtimes, a more general degree-runtime correspondence is postulated by~\cite[Hypothesis~2.1.5]{hopkins-thesis}, namely, degree $D$ corresponds to runtime roughly $n^{\tilde{\Theta}(D)}$ or equivalently $\exp(\tilde{\Theta}(D))$, which is roughly the number of terms in the polynomial. Here, $\tilde{\Theta}(\cdot)$ hides a $\polylog(n)$ factor. Notably, degree $n^\delta$ for a constant $\delta > 0$ corresponds to runtime $\exp(n^{\delta \pm o(1)})$.

\begin{hyp}[Informal]
\label{hyp:higher-D}
For ``natural'' high-dim stat problems, for some large enough constant $C > 0$ and any growing sequence $D = D_n$,
\[ \emph{$\exp(O(D/\log^C n))$-time algorithms} \;\subseteq\; \emph{degree-$D$ polynomials} \;\subseteq\; \emph{$n^{O(D)}$-time algorithms} \]
where ``$A \subseteq B$'' means class $B$ is at least as powerful as class $A$.
\end{hyp}

\noindent Here it is the second ``$\subseteq$'' that is justified by naive term-by-term evaluation of a polynomial (see Section~\ref{sec:eval-poly}), while the first ``$\subseteq$'' is more subtle and will need to be justified later (see Section~\ref{sec:list}).

\subsection{Defining ``Success'' for Detection and Recovery}
\label{sec:def-success}

We hope to quantify what polynomial degree is required to solve a given statistical task. Given that a polynomial may not necessarily output a binary value, how exactly should we define ``success'' for a low-degree algorithm? We now address this for the detection and recovery tasks, with discussion of other tasks (such as refutation and optimization) deferred to Section~\ref{sec:tasks}. Here we present the definitions, and in Section~\ref{sec:commentary} we include a deeper discussion and address some common questions.

\paragraph{Detection.} For detection, we will in general consider two sequences of distributions $\PP = (\PP_n)_{n \ge 1}$ and $\QQ = (\QQ_n)_{n \ge 1}$, where $\PP_n$ and $\QQ_n$ are distributions over (a subset of) $\RR^N$ for some $N = N_n$. For instance, $N = \binom{n}{2}$ in the case of planted clique. As usual, consider the limit $n \to \infty$, where other parameters such as $k$ will scale with $n$ in some prescribed way. For simplicity we will often omit ``sequence of,'' for instance, speaking of a polynomial $f: \RR^N \to \RR$ when we really mean a sequence $f = f_n$ of polynomials, one for each problem size.

We first formally define the detection task in both its ``strong'' and ``weak'' variants.

\begin{definition}\label{def:det}
A {\bf\emph{test}} is a function $t: \RR^N \to \{0,1\}$ that takes in a sample from either $\PP$ or $\QQ$ and outputs a ``guess'': $1$ for $\PP$, or $0$ for $\QQ$. (More formally, a test is a sequence of such functions, one for each value of $n$.)
\begin{itemize}
\item We say {\bf\emph{strong detection}} is achieved if the guess is correct with high probability, that is,
\[ \Pr_{Y \sim \PP}(t(Y) = 0) + \Pr_{Y \sim \QQ}(t(Y) = 1) = o(1) \qquad \text{as } n \to \infty. \]
\item {\bf\emph{Weak detection}} means non-trivial advantage over a random guess, that is,
\[ \Pr_{Y \sim \PP}(t(Y) = 0) + \Pr_{Y \sim \QQ}(t(Y) = 1) = 1 - \Omega(1) \qquad\text{as } n \to \infty. \]
\end{itemize}
\end{definition}

\noindent We will use the following analogous notion of success for polynomial tests.

\begin{definition}\label{def:sep}
A polynomial $f: \RR^N \to \RR$ is said to {\bf\emph{strongly separate}} distributions $\PP$ and $\QQ$ if
\[ \sqrt{\max\left\{\Vop_{Y \sim \PP}[f(Y)], \, \Vop_{Y \sim \QQ} [f(Y)]\right\}} = o\left(\left|\Eop_{Y \sim \PP} [f(Y)] - \Eop_{Y \sim \QQ} [f(Y)]\right|\right), \]
and is said to {\bf\emph{weakly separate}} $\PP$ and $\QQ$ if
\[ \sqrt{\max\left\{\Vop_{Y \sim \PP}[f(Y)], \, \Vop_{Y \sim \QQ} [f(Y)]\right\}} = O\left(\left|\Eop_{Y \sim \PP} [f(Y)] - \Eop_{Y \sim \QQ} [f(Y)]\right|\right). \]
\end{definition}

Let's assume $\PP$ and $\QQ$ have finite moments of all orders so that all the quantities above are guaranteed to exist. The significance of the above definition is as follows. Suppose $f$ strongly separates $\PP,\QQ$, and suppose we are able to compute the value $f(Y)$. Then by Chebyshev's inequality, we can achieve strong detection by thresholding $f(Y)$ at the appropriate cutoff, namely the midpoint between the two expectations. Similarly, weak separation implies that weak detection can be achieved by thresholding $f(Y)$, although the right threshold might not be the midpoint anymore~\cite[Prop~3.2]{ranked}. In other words, strong (or weak) separation provides a natural sufficient condition for strong (respectively, weak) detection by polynomials, based on the first two moments. For a given detection problem, we will aim to prove both upper and lower bounds on the polynomial degree $D = D_n$ required in order to have strong (or weak) separation by some degree-$D$ polynomial. If separation fails for some $D$, we will interpret this as suggesting hardness of detection by algorithms of ``less powerful'' runtime as per Section~\ref{sec:d-r-corr}. For instance, if we want to argue that poly-time algorithms cannot achieve strong (or weak) detection, we will aim to rule out strong (respectively, weak) separation by degree-$D$ polynomials for some $D = \omega(\log n)$, or ideally a larger $D$ to be safe.

\paragraph{Recovery.} Moving on to recovery, our observation $Y$ now comes from the distribution $\PP = (\PP_n)_{n \ge 1}$, and we will aim to estimate some scalar value $x$. For instance, in the case of planted clique a natural choice is $x := \One_{1 \in S}$ where $S$ is the set of clique vertices, as discussed above. With some abuse of notation, we write $(x,Y) \sim \PP$ for the joint distribution of $(x,Y)$ under the planted model. The following notion of success was first defined in~\cite{SW-estimation}. We write $\RR[Y]_{\le D}$ for the multivariate polynomials in $Y$ of degree at most $D$.

\begin{definition}\label{def:mmse}
The {\bf\emph{degree-$D$ minimum mean squared error}} is
\[ \MMSE_{\le D} := \inf_{f \in \RR[Y]_{\le D}} \Eop_{(x,Y) \sim \PP} [(f(Y)-x)^2]. \]
\end{definition}

\noindent For a given recovery problem, we will aim to characterize the behavior of $\MMSE_{\le D}$ for various values of $D$. The trivial estimator $f(Y) \equiv \EE[x]$ has mean squared error $\Var(x)$, so in the ``hard'' regime of parameters we will often aim to prove that low-degree polynomials cannot significantly beat this: $\MMSE_{\le D} \ge (1-o(1)) \Var(x)$, say for some $D = \omega(\log n)$. 

Why do we define ``success'' for detection and recovery in these particular ways? One quick answer is that these definitions hit a ``sweet spot'' where low-degree algorithms do tend to achieve them in the ``easy'' regime, and we also have mathematical tools for proving that no low-degree algorithm can achieve them in the ``hard'' regime. More detailed commentary on the definitions is deferred to Section~\ref{sec:commentary}.

\paragraph{Proof techniques.} We now touch briefly on the most basic techniques for proving low-degree lower bounds, with a more complete treatment deferred to Section~\ref{sec:techniques}. For now, we focus on detection, which tends to be the easiest setting to analyze. Suppose we aim to rule out strong detection by any degree-$D$ polynomial. By rewriting the definition, this is equivalent to showing that the following ratio is $O(1)$:
\[ \sup_{f \in \RR[Y]_{\le D}} \frac{\Eop_{Y \sim \PP} [f(Y)] - \Eop_{Y \sim \QQ} [f(Y)]}{\sqrt{\max\left\{\Vop_{Y \sim \PP}[f(Y)], \, \Vop_{Y \sim \QQ} [f(Y)]\right\}}}. \]
We will relax this to a larger quantity that is more tractable to analyze, and aim to show that even this larger quantity is $O(1)$. Specifically, we drop the variance under $\PP$ from the denominator and also take $\Eop_{Y \sim \QQ} [f(Y)] = 0$ without loss of generality, to arrive at the \emph{advantage},
\begin{equation}\label{eq:adv}
\Adv_{\le D} := \sup_{f \in \RR[Y]_{\le D}} \frac{\EE_{Y \sim \PP}[f(Y)]}{\sqrt{\EE_{Y \sim \QQ}[f(Y)^2]}}.
\end{equation}
This is a key quantity that goes by a few different names. For example, it can also be denoted $\|L^{\le D}\|$ due to its characterization as the \emph{norm of the low-degree likelihood ratio}, and is also closely related to the \emph{chi-squared divergence}. We will revisit these viewpoints in Section~\ref{sec:adv}.

In light of the above, to rule out strong separation it suffices to show $\Adv_{\le D} = O(1)$. Crucially, $\Adv_{\le D}$ admits an explicit formula,
\[ \Adv_{\le D}^2 = \sum_{i=0}^m \left(\EE_{Y \sim \PP}[h_i(Y)]\right)^2, \]
where $h_0,\ldots,h_m$ are any choice of orthonormal polynomials with respect to $\QQ$ that form a basis for $\RR[Y]_{\le D}$. See Proposition~\ref{prop:adv-c} for further details and the proof. As long as $\QQ$ has independent coordinates --- as in $G(n,1/2)$ for instance --- it is often tractable to construct orthogonal polynomials and compute $\Adv_{\le D}$ explicitly using the above formula. This was a powerful observation that fueled the initial wave of success for the low-degree framework. The idea of proving low-degree lower bounds in this way first appeared in two concurrent works~\cite{sos-detect,HS-bayesian}, both inspired by~\cite{pcal}, and the framework was further refined by~\cite{hopkins-thesis}.

\subsection{What Do We Hope to Prove?}
\label{sec:hope}

Having defined ``success'' for a low-degree polynomial algorithm, we can now illustrate the type of results we might aim to prove for a given statistical task. As usual, we'll start with the planted clique example, considering both detection and recovery. We recall here the setting.

\begin{definition}[Planted Clique, Binomial\footnote{The clique size concentrates near $k$ rather than being exactly $k$. This is not expected to make a difference, but makes the analysis easier. Variations of the planted clique problem are discussed in~\cite{clique-equiv}.} Variant]
\label{def:pc}
The input is a random graph, represented as an element $Y \in \{0,1\}^{\binom{n}{2}}$.
\begin{itemize}
    \item Under the null distribution $\QQ$, $Y \sim G(n,1/2)$.
    \item Under the planted distribution $\PP$, each vertex is included in the set $S$ with probability $k/n$ independently, and the observed graph $Y$ is the union of $G(n,1/2)$ with a clique on $S$.
    \item For the recovery task, the quantity to estimate given $Y \sim \PP$ is $x := \One_{1 \in S}$, the indicator for vertex 1's presence in the clique.
\end{itemize}
\end{definition}

\noindent The following result shows that $k \approx \sqrt{n}$ is the threshold for both detection and recovery by $O(\log n)$-degree polynomials. This matches the conjectured computational threshold.

\begin{theorem}[Low-Degree Analysis for Planted Clique]
\label{thm:clique}
Consider the planted clique model as in Definition~\ref{def:pc}, with $k = \Theta(n^a)$ for a constant $a \in (0,1)$.
\begin{enumerate}
    \item (Hard regime) If $a < 1/2$ and $D = o\left(\frac{\log n}{\log\log n}\right)^2$ then
    \begin{itemize}
        \item no degree-$D$ polynomial weakly separates $\PP$ and $\QQ$; and
        \item $\MMSE_{\le D} \ge (1-o(1)) k/n$.
    \end{itemize}
    \item (Easy regime) If $a > 1/2$ then
    \begin{itemize}
        \item there is a degree-$1$ polynomial that strongly separates $\PP$ and $\QQ$; and
        \item for any constant $c > 0$ there is a constant $C = C(a,c) > 0$ such that $\MMSE_{\le C} \le n^{-c}$.
    \end{itemize}
\end{enumerate}
\end{theorem}

\begin{proof}
The lower bound for detection comes from~\cite[Lemma~2.4.1]{hopkins-thesis}, based on calculations that were implicit in~\cite{pcal}. The improved parameters here are based on~\cite[Prop~4.9]{coloring-clique}. The lower bound for recovery is~\cite[Thm~3.7]{SW-estimation}. For the upper bounds, the degree-1 polynomial for detection is simply the total edge count, and the result for recovery is based on~\cite[Thm~4.8]{SW-estimation}.
\end{proof}

\noindent Below, we make some comments to aid the interpretation of Theorem~\ref{thm:clique}.
\begin{itemize}
    \item {\bf Value of the MMSE}: The ``trivial'' MSE (mean squared error) is $\Var(x) = k/n (1 - k/n) = (1-o(1)) k/n$, while an MSE value of $o(1/n)$ guarantees (by thresholding the estimator) exact recovery with high probability. Thus, low-degree MMSE transitions from near-trivial to near-perfect at $k \approx \sqrt{n}$. We note that $\MMSE_{\le D}$ is directly related to the \emph{vector MMSE}, where we aim to estimate $(\One_{1 \in S}, \One_{2 \in S}, \ldots, \One_{n \in S})$ using a vector of polynomials~\cite[Section~1]{SW-estimation}.
    \item {\bf Value of $D$ in the hard regime}: While degree $O(\log n)$ (or slightly higher, $\log^{1+\eps} n$~\cite{hopkins-thesis}) has become the standard benchmark for probing poly-time computation, in the hard regime it is desirable to rule out polynomials of even higher degree if possible, just to be safe, and to understand the true ``degree complexity.'' The lower bound above rules out polynomials of degree roughly $\log^2 n$, which should be nearly optimal: planted clique has an algorithm of quasipolynomial time $n^{O(\log n)}$ based on a brute-force search for cliques of size $O(\log n)$, and this should be captured by a polynomial of degree roughly $\log^2 n$ (the number of edges in such a clique). In contrast to planted clique, many other problems have a more dramatic phase transition, requiring degree $n^{\Omega(1)}$ in the hard regime. In this sense, planted clique is among the easiest of the hard problems, which may explain its success as a starting point for reductions (see Section~\ref{sec:reductions}).
    \item {\bf Value of $D$ in the easy regime}: We have viewed the threshold on a coarse scale $k = \Theta(n^a)$ for a constant $a \ne 1/2$. The upper bounds were quite simple, requiring only constant degree, since we are actually deep into the easy regime by a factor of $n^{\Omega(1)}$. We may want to zoom in and investigate the degree complexity closer to the threshold, say when $k = c\sqrt{n}$ for a constant $c > 0$. It seems this has not entirely been worked out for the clique problem, but there are other settings where degree $\Theta(\log n)$ is known to be both necessary and sufficient for recovery in the easy regime, e.g.~\cite[Thm~2.6]{sharp-est}.
    \item {\bf Detection-recovery gaps}: The detection and recovery thresholds match for planted clique, but need not match in general. One example of such a \emph{detection-recovery gap} occurs in the \emph{planted dense subgraph} problem, a generalization of planted clique where the clique is replaced by an Erd\H{o}s--R\'{e}nyi graph of density higher than the ambient graph (see~\cite{det-rec-reduction}). In some regimes, detection becomes easier than recovery due to a trivial test based on the total edge count. Here, the low-degree thresholds for detection and recovery are different, and coincide with the (best known) computational thresholds for detection and recovery, respectively~\cite{SW-estimation,subhypergraph}. See for instance~\cite[Thm~2.7]{dense-cycles} for another example of a detection-recovery gap, with a side-by-side comparison of the detection and recovery results.
    \item {\bf Upper bounds}: While our main focus is on hardness results, the low-degree upper bounds play an important supporting role. Lower bounds are most meaningful when we know the same notion of success being ruled out is actually achievable in the easy regime. We revisit this in Section~\ref{sec:comm-upper}.
\end{itemize}

\noindent A few additional points are illustrated by the next example, sparse PCA detection.

\begin{definition}[Sparse PCA Detection, Binomial Variant]
\label{def:spca}
The input is a matrix $Y \in \RR^{n \times N}$.
\begin{itemize}
    \item Under the null distribution $\QQ$, $Y$ is i.i.d. $\cN(0,1)$.
    \item Under the planted distribution $\PP$, first a hidden vector $x \in \RR^n$ is generated with i.i.d. entries of the form
    \[ x_i = \begin{cases}1/\sqrt{k} & \text{with probability } k/(2n), \\ -1/\sqrt{k} & \text{with probability } k/(2n), \\ 0 & \text{with probability } 1-k/n. \end{cases} \]
    Then each column of $y$ is drawn independently from $\cN(0,I+\beta xx^\top)$.
\end{itemize}
Suppose $\beta > 0$ is a constant and $n/N \to \gamma$ for a constant $\gamma > 0$. Define the SNR parameter $\lambda := \beta/\sqrt{\gamma}$, which is also a constant.
\end{definition}

What are the best known algorithms for this detection task? We focus on strong detection here, since weak detection is easy for any constant $\lambda > 0$ using the trace of the sample covariance matrix. Let's also assume $1 \ll k \ll n$, which makes strong detection information-theoretically possible. The following two algorithmic results are known:
\begin{itemize}
\item When $\lambda > 1$, strong detection becomes easy (in poly time) by thresholding the maximum eigenvalue of the sample covariance matrix. This is the ``BBP'' transition in random matrix theory~\cite{BBP,BS-spiked}, named after Baik, Ben Arous, and P{\'e}ch{\'e}.
\item When $\lambda < 1$, the best known algorithms have runtime $\exp(\tilde{O}(k^2/n)) + \poly(n)$~\cite{subexp-sparse,anytime-pca}, with poly-time algorithms appearing when $k \lesssim \sqrt{n}$~\cite{AW-sparse,cov-thresholding}.
\end{itemize}
The lower bound below shows that the combination of these two algorithms cannot be improved within the low-degree framework.

\begin{theorem}[Low-Degree Analysis for Sparse PCA]
Consider the sparse PCA detection problem as defined in Definition~\ref{def:spca}. If
\[ \lambda < 1 \qquad\text{and}\qquad D = o(k^2/n) \]
then no degree-$D$ polynomial strongly separates $\PP$ and $\QQ$.
\end{theorem}

\begin{proof}
This follows from~\cite[Thm~2.14]{subexp-sparse}. It should also be possible to recover this using a simpler proof strategy of~\cite{sparse-clustering}. See also~\cite{sparse-adversarial} for a refinement of the result.
\end{proof}

Recalling Hypothesis~\ref{hyp:higher-D} --- namely that degree-$D$ polynomials are expected to be comparable in power to all $\exp(\tilde{\Theta}(D))$-time algorithms --- the bound on $D$ tracks the best known runtime as expected.

\begin{remark}[Sharp Thresholds vs Smooth Tradeoffs]
\label{rem:sharp-smooth}
The sparse PCA example simultaneously illustrates two different types of computational phase transitions. First, we have a sharp threshold where the difficulty of the problem changes abruptly at $\lambda = 1$. Below this threshold, we have a smooth tradeoff where the runtime varies smoothly as a function of the parameter $k$.

Other notable phase transitions can be classified as one of these types. In the planted clique model, the computational ``threshold'' at $k \approx \sqrt{n}$ is actually a smooth tradeoff, where it is possible to make headway into the hard regime at the expense of increasing the runtime in a smooth fashion~\cite{alon-clique}. Another smooth tradeoff occurs in tensor PCA, and the degree complexity~\cite{ld-notes} tracks the best known super-polynomial runtime~\cite{strongly-refuting-csp,sos-sphere-1,sos-sphere-2,kikuchi}. On the other hand, the stochastic block model has a sharp threshold known as the \emph{Kesten--Stigum (KS) bound}~\cite{decelle} (see also~\cite{AS-acyclic}), where the runtime abruptly jumps from polynomial to (conjecturally) fully exponential $\exp(n^{1-o(1)})$~\cite[Thm~2.20]{spectral-planting}. The KS bound was the first sharp threshold that was shown to be captured by a low-degree polynomial phase transition~\cite[Thm~1.9]{HS-bayesian}, with some aspects later improved and extended~\cite{hopkins-thesis,spectral-planting,kunisky-coordinate-2,conj-sharp-sbm,sharp-est}.
\end{remark}

We now conclude Part~I, having seen the basics of how low-degree polynomials can be used probe the computational complexity of statistical tasks. This framework allows us to form conjectures and rigorously substantiate them by establishing algorithmic barriers, at least for some class of approaches. Remarkably, degree complexity appears to reliably track time complexity for a wide variety of different problems, with a list of specific examples appearing in Section~\ref{sec:list}. There are also limits to the types of problems where this framework is appropriate, and this is the topic of Section~\ref{sec:conj}.

\newpage
\phantomsection
\begin{center}
\huge{\bf Part II: Deeper Discussion}
\end{center}
\addcontentsline{toc}{section}{Part II: Deeper Discussion}

While Part I introduced the basic framework, Part II is intended as a deeper discussion of the surrounding philosophy, as well as extensions of the basic framework and connections to other frameworks. The Sections~\ref{sec:commentary}--\ref{sec:open} that comprise Part II are somewhat modular and can be read in any order. See Section~\ref{sec:org} for an overview of what each one contains.

\section{Commentary on the Definitions}
\label{sec:commentary}

This section is devoted to addressing some frequently asked questions in regards to Part~I.

\subsection{What model of computation is used when referring to runtime?}

Above we have referred to ``polynomial-time'' algorithms, meaning the runtime is $\poly(n)$ for our notion of problem size $n$. The number of input variables will always be $n^{\Theta(1)}$, so this is equivalent to requiring the runtime to be polynomial in the input size. We allow randomized algorithms, in which case the expected runtime should be $\poly(n)$.

To make all of this formal, one needs to specify a model of computation. For a discrete problem like planted clique, a Turing machine would be appropriate, with the input encoded as a binary string of length $\binom{n}{2}$. For, say, sparse PCA, the inputs are real numbers, which would need to be rounded off to some precision. It is somewhat commonplace to ignore these issues of numerical precision and assume that basic arithmetic operations can be done exactly on real-valued inputs. The \emph{real RAM model} of Blum--Shub--Smale provides one way to formalize this~\cite{BSS}. Ultimately, I believe the conclusions discussed in this survey are not particularly sensitive to the choice of computation model (within reason) and there should be no inherent issues with approximating the algorithms we'll discuss in floating-point arithmetic.

Maybe one consideration to have in mind is that the computation of eigenvalues and eigenvectors of a matrix cannot be done exactly (even in the real RAM model) and must be approximated iteratively. While this is normally thought of as efficiently computable, there is no guarantee it can be done in polynomial time if the eigenvalues are extremely close together. The typical situation we will encounter is one with a ``spectral gap'' between the two largest (in magnitude) eigenvalues and we aim to compute the largest one (and its associated eigenvector), which can be done by power iteration as explained in Remark~\ref{rem:log-deg}.

In any case, the model of computation will only be relevant for our conjectural discussions about time complexity. We have formally defined a low-degree polynomial ``model of computation'' and this will be our primary focus.

\subsection{Can low-degree polynomials be computed efficiently?}

Not necessarily, for a few reasons. First is the issue of computing the coefficients. When we speak of a degree-$D$ polynomial for some $D = D_n$, we really mean a sequence of polynomials $f = f_n$, one for each problem size. The coefficients are allowed to have arbitrary dependence on $n$ (and other model parameters), so in principle they may be hard to compute or even uncomputable. Polynomials should be thought of as a \emph{non-uniform} model of computation as is standard in circuit complexity (and the related complexity class P/poly). In other words, we allow a different algorithm for each $n$ so that the polynomial coefficients (as well as the value at which to threshold the polynomial) can be hard-coded into the algorithm.

The second issue is that we tend to work with degree-$O(\log n)$ polynomials, which have a super-polynomial number of terms, and thus cannot necessarily be computed in poly time even if the coefficients are ``easy.'' While some natural degree-$O(\log n)$ polynomials can be approximately evaluated in poly time (see Section~\ref{sec:alg-captured}), there is no such guarantee in general.

Ultimately, our focus here is on lower bounds, where the relevant hypothesis is that degree-$O(\log n)$ polynomials are \emph{at least} as powerful as poly-time algorithms. Our main concern will not be ``can polynomials be computed efficiently?''\ but rather ``can important classes of efficient algorithms be represented as polynomials?''

\subsection{Why ``separation'' as the notion of success for low-degree testing?}
\label{sec:comm-sep}

In my experience, the notion of separation tends to attract some criticism, and I will take this opportunity to expand on the justification for posing the framework in this way.

Backing up a bit, what do we desire from a definition of ``success'' for low-degree tests\footnote{Not to be confused with the ``low-degree test'' in property testing~\cite{ld-test}.} in the first place? The hope is to establish a unified notion of success that can be used for both upper and lower bounds across various high-dim stat problems. There are a few criteria we'd like this to meet: (i) for the upper bounds to be meaningful, ``success'' should, at a minimum, imply that detection is information-theoretically possible; (ii) for the lower bounds to be meaningful, we need to believe that whenever there is an efficient algorithm (perhaps within some class) for detection, there should also be a low-degree polynomial that achieves ``success''; (iii) we need to actually have the tools to prove upper and lower bounds for this notion of success. Based on these criteria, I will argue that separation is the most satisfying definition of success that we have at the moment.

One natural attempt at defining success would be to require that detection (weak or strong) can be achieved by thresholding the value of some polynomial. In other words, the class of tests is restricted to \emph{polynomial threshold functions (PTFs)}. Unfortunately this currently fails criterion~(iii), that is, we have not managed to prove lower bounds against PTFs for the types of problems discussed in this survey. Our current low-degree lower bounds are based on computing $\Adv_{\le D}$ (or some variant thereof) as defined in~\eqref{eq:adv}, which rules out separation but does not rule out thresholding. In other words, they rule out the \emph{standard analysis} of PTFs via Chebyshev's inequality, but do not formally rule out PTFs. It is an important open question to rule out PTF tests.

In the earlier days of the low-degree framework, I feel it was common to treat $\Adv_{\le D}$ itself as the notion of success, where $\Adv_{\le D} = O(1)$ means the problem is predicted to be hard while $\Adv_{\le D} = \omega(1)$ suggests that it should be easy. However, $\Adv_{\le D} = \omega(1)$ does not imply that detection can be achieved by thresholding a polynomial (since there is no control over the variance under $\PP$), and need not even imply that detection is information-theoretically possible, so this fails criterion~(i). More recently we have discovered some natural testing problems where $\Adv_{\le D} = \omega(1)$ even in part of the ``hard'' regime~\cite{fp,grp-testing,subhypergraph,correlated-er}. Despite this, the threshold for \emph{separation} does match the true (conjectured) computational threshold in these examples, and this began the trend of taking separation as the explicit definition of success in low-degree upper and lower bounds. In some cases, lower bounds against separation require calculating a ``conditional'' $\Adv_{\le D}$ with a modified $\PP$ (see Section~\ref{sec:cond}) but this is only a proof device: the definition of separation still pertains to the \emph{original} $\PP$ and $\QQ$.

Separation meets criteria~(i) and~(iii), but what about~(ii)? As with any notion of success, this is something we can only hope to amass evidence for by studying many examples: to make our lower bounds against separation meaningful, we'd like to see that for many problems, in the ``easy'' regime where poly-time algorithms are known, there is some low-degree polynomial that succeeds \emph{specifically in the sense of separation}. There are by now many low-degree upper bounds of this form, and some are listed in Section~\ref{sec:list}.

I am not aware of any ``natural'' problems where separation behaves differently from PTFs. One can certainly imagine designing a badly-behaved polynomial that succeeds by thresholding but not separation in the easy regime, but there tends to be a \emph{different} polynomial of comparable degree that does achieve separation. One reason to expect thresholding and separation are connected is that natural distributions tend to be \emph{hypercontractive} --- which roughly means the moments of low-degree polynomials are well-behaved --- and hypercontractivity along with bounded $\Adv_{\le D}$ does rule out thresholding subject to some technical conditions\footnote{The most serious technical condition here is that we only rule out thresholding with a specific high success probability, not any $1-o(1)$.}~\cite[Thm~4.3]{ld-notes}.

As a final note about separation, I find that detection lower bounds against separation attract more criticism than recovery lower bounds involving $\MMSE_{\le D}$. However, I consider the two notions to be highly analogous. Strong separation between $\PP,\QQ$ is equivalent to having mean squared error $o(1)$ in the following estimation setting: the estimand $x$ is drawn uniformly from $\{0,1\}$ and then we receive a sample $Y$ drawn from either $\QQ$ or $\PP$ if $x$ is $0$ or $1$ respectively. Relatedly, $\MMSE_{\le D}$ suffers from the same flaw as separation: if we show $\MMSE_{\le D}$ is large, there could still conceivably be a degree-$D$ polynomial that one can threshold to achieve exact recovery with high probability. You might view this as either a point in favor of separation or a point against $\MMSE_{\le D}$.

\subsection{What's the point of low-degree \emph{upper} bounds?}
\label{sec:comm-upper}

If we want to argue that a problem is \emph{easy} in some regime, we should just find a poly-time algorithm and prove that it works, without any need for low-degree polynomials. However, ideally we should also establish a \emph{low-degree} upper bound in the easy regime, meaning we prove that some (say) degree-$O(\log n)$ polynomial succeeds, \emph{in the precise sense that we've defined success} (separation or $\MMSE_{\le D}$). This serves a different purpose from the poly-time algorithm, and the two need not be related. The point of low-degree upper bounds is to serve as a ``sanity check'' for the lower bounds: without them, we might worry that low-degree polynomials fail in both the hard and easy regimes. Only with matching lower and upper bounds, both using the same notion of success, can we rigorously establish that a phase transition has occurred. Put another way, the upper bounds exhibit that degree-$O(\log n)$ polynomials are powerful, making it more meaningful when we show their failure in the opposite regime. The upper bounds are especially important when considering a new style of problem or a new notion of success\footnote{One notion of success that has been used to state lower bounds but currently lacks the matching upper bounds is the minimax framework as in~\cite[Thm~1]{graphon}. It is an interesting question to investigate whether or not low-degree polynomials can achieve this notion of success in the easy regime.} that qualitatively deviates from existing results.

Currently we are (embarrassingly) missing some low-degree upper bounds even in the very basic settings from Section~\ref{sec:hope}. Not all of these are necessarily difficult, but they have never been written down, as unfortunately they can be somewhat tedious to prove. For planted clique, it would be nice to see that degree-$O(\log^2 n)$ polynomials succeed (in terms of both strong separation and $\MMSE_{\le D}$) in the hard regime and that degree-$O(\log n)$ polynomials succeed when $k = \epsilon \sqrt{n}$. For sparse PCA, it would be nice to see that degree-$O(\log n)$ polynomials succeed when $\lambda > 1$ (this one should be similar to~\cite[Section~2.6]{HS-bayesian}) and that degree-$\tilde{O}(k^2/n)$ polynomials succeed in general for $k \gtrsim \sqrt{n}$. Every low-degree upper bound adds to the credibility of the low-degree framework as a whole, and I hope to see more of them appear in the future. In Section~\ref{sec:alg-captured} we will discuss some popular types of algorithms that can typically be made into low-degree upper bounds, but usually some problem-specific work is needed for the analysis.

\section{Tasks: Detection, Recovery, and More}
\label{sec:tasks}

For models like those described in Section~\ref{sec:high-dim} there are a variety of different objectives we might be interested in. In Part I we have already mentioned the detection and recovery tasks. Here we will define two additional tasks, refutation and optimization, and explore the relations among all four tasks. In general, all these tasks may have different computational thresholds. A major strength of the low-degree framework is its ability to probe all these different tasks, and differentiate between the corresponding thresholds.

\subsection{Definitions}
\label{sec:task-def}

We will use the planted clique model as a running example to illustrate some different statistical tasks, but the analogous tasks can be defined for other models too. Recall from Section~\ref{sec:mot-ex} that the planted clique distribution (with parameters $n,k$) is generated by sampling a random graph $G(n,1/2)$ and then planting a clique on $k$ randomly selected vertices. We will call this distribution $\PP = \PP_n$ for ``planted,'' where $k$ will scale with $n$ in some prescribed way. We will also refer to the ``null'' distribution $\QQ = \QQ_n = G(n,1/2)$, which represents pure ``noise'' without planted structure.

\begin{itemize}
    \item {\bf Detection (a.k.a.\ Hypothesis Testing)}: This is the task of distinguishing $\PP$ from $\QQ$, given a sample from either (in this case, a random graph drawn either from $\PP$ or $\QQ$). The strong and weak versions of detection were defined in Definition~\ref{def:det}. In general, our null distribution will usually be a product measure (i.e., independent entries, which is the case for $G(n,1/2)$) and we'll see in Section~\ref{sec:techniques} that this is helpful for making the analysis easier. An exception is ``planted-vs-planted testing''~\cite{planted-v-planted} discussed in Section~\ref{sec:planted-v-planted-pf}, e.g., the task of distinguishing the case of 1 planted clique from 2 planted cliques~\cite{coloring-clique}.
    \item {\bf Recovery (a.k.a.\ Estimation)}: Given a graph drawn from $\PP$, this is the task of recovering the $k$ vertices on which the clique was planted. \emph{Exact recovery} refers to exactly identifying the set of $k$ vertices with high probability, i.e., probability $1-o(1)$. Weaker notions of approximate recovery also exist, which may be problem-specific. ``Estimation'' usually refers to approximate recovery of some hidden vector (e.g., the indicator vector for the clique vertices) with respect to some error metric (e.g., mean squared error).
    \item {\bf (Non-Planted) Optimization}: Given a graph drawn from $\QQ$, this is the task of finding as large a clique as possible. Success means finding a $k$-clique with high probability, for some given parameter $k$. Note that the input has no planted structure, so there may be many $k$-cliques in the graph and we are satisfied with finding any one of them. This is in contrast to recovery, where our goal is to find \emph{the} planted clique.
    \item {\bf Refutation (or Certification)}: This is a more subtle task that can be skipped on a first reading. Given a graph drawn from $\QQ$, refutation is the task of (algorithmically) \emph{proving} that the planted structure (in this case, a $k$-clique) does not exist. Formally, a refuter takes a graph $Y$ as input and outputs NO or MAYBE with the following two properties: (1) for \emph{any} input $Y$ that contains a $k$-clique, the output must be MAYBE; (2) with high probability over $Y \sim \QQ$, the output must be NO. Here, the response NO means that the existence of a $k$-clique has been ruled out with absolute certainty for this particular input $Y$. We can only hope for a refuter when $k$ is large enough so that $Y \sim \QQ$ has no $k$-clique with high probability. The more general term ``certification'' refers to proving that the input has some property of interest.
\end{itemize}

\noindent For all these tasks, we want computationally efficient algorithms that work for the widest possible range of $k$ values. The first three become more difficult as $k$ decreases but optimization becomes more difficult as $k$ increases. These four types of tasks (and variations thereof) can be defined for many other high-dim stat problems.\footnote{Optimization and refutation are perhaps not really ``statistical'' problems, but terms like ``high-dim stats'' or ``statistical-computational gap'' are still sometimes used here. Alternative terms like ``high-dimensional average-case problems'' or ``optimization-computation gap'' have also been used, and may be more appropriate.} We next briefly motivate each of the tasks, speaking now in high generality (beyond the clique problem). Recovery and optimization are perhaps the most natural, and can be motivated by real-world applications: determining which genes cause a disease would be a (planted) recovery problem because there is a ``ground truth,'' whereas finding a good airline schedule would be a (non-planted) optimization problem. Detection can also be intrinsically motivated, and is also often used as a device for reasoning about recovery, since detecting the presence of a planted signal appears to be a necessary precursor to actually finding it. Certification is motivated by a desire to be absolutely certain that some property holds on the given input. The certification task can also be a useful device for indirectly reasoning about limitations of convex programs --- which are a natural approach to certification --- as in~\cite{sk-cert}.

In the case of planted clique, where are the computational thresholds believed to lie for the various tasks? In Part I we saw that the best known poly-time algorithms for detection and recovery require $k \gtrsim \sqrt{n}$. One way to solve detection and recovery when $k \gtrsim \sqrt{n}$ is to use the leading eigenvalue or eigenvector (respectively) of the signed adjacency matrix $A$, where $A_{ij} = 1$ whenever edge $(i,j)$ is present, $A_{ij} = -1$ whenever the edge is absent, and $A_{ii} = 0$.

The best known algorithms for refutation also require $k \gtrsim \sqrt{n}$, and this can again be achieved using the maximum eigenvalue of $A$. Any graph with a $k$-clique must have maximum eigenvalue at least $k-1$, since this is the value of $u^\top A u / \|u\|^2$ where $u$ is the indicator vector for the clique. Therefore, if the maximum eigenvalue is below $k-1$, this refutes the existence of a $k$-clique. For $G(n,1/2)$, the maximum eigenvalue of $A$ is $O(\sqrt{n})$ with high probability, meaning the above strategy for refutation works when $k \gtrsim \sqrt{n}$.

Optimization, on the other hand, is interesting for much smaller $k$ values. The maximum clique in $G(n,1/2)$ has size $(2 \pm o(1)) \log_2 n$ with high probability~\cite[Section~4.5]{prob-method-book}. The influential work of Karp in the 70's~\cite{karp} showed that a simple greedy algorithm can find a clique of ``half-optimal'' size $(1-\epsilon) \log_2 n$ with high probability, and famously asked whether any poly-time algorithm can improve this threshold. To date, no better algorithm has been found, and low-degree lower bounds suggest the gap is fundamental: low-degree polynomials fail to surpass $(1+\epsilon) \log_2 n$, appropriately defined~\cite{indep-set,strong-ld}. Unfortunately the matching low-degree \emph{upper bound} is still missing for this particular problem, although a variation for sparse graphs does have a matching upper bound~\cite{indep-set}.

\subsection{Relations Among the Tasks}
\label{sec:relations-tasks}

What formal relations exist between the four tasks above? Optimization and refutation are ``dual'' to each other in the sense that optimization seeks to (constructively) prove a lower bound on the maximum clique size of a graph drawn from $\QQ$, while refutation seeks to prove an upper bound on the same quantity. However, these two tasks need not have similar difficulty, as seen for example in~\cite{sk-cert,affine-planes,sos-amp-robustly}. In fact, optimization does not appear to have any formal relation to the other three tasks, as it is only interesting when $k$ lies below the maximum clique size in $G(n,1/2)$ --- roughly $2\log_2 n$ --- while the other tasks are only interesting when $k$ lies above this threshold.

Detection, recovery, and refutation are related as shown in Figure~\ref{fig:relations}, namely recovery and refutation are each at least as hard as detection but incomparable to each other. This is justified by the following reductions which we illustrate using the clique example, but the arguments are simple and quite generic.

\begin{figure}
    \centering
    \begin{tikzcd}
    &
        \begin{tabular}{c}
             Strong Detection  \\
              ($\mathbb{P}$ vs $\mathbb{Q}$)
        \end{tabular}
    & \\
    \begin{tabular}{c}
             Exact Recovery  \\
              (in $\mathbb{P}$)
    \end{tabular}
        \arrow{ur}{}
        & &
    \begin{tabular}{c}
             Refutation  \\
              (in $\mathbb{Q}$)
        \end{tabular}
        \arrow{ul}{}
    \end{tikzcd}
    \captionsetup{width=.95\linewidth}
    \caption{Formal relations among detection, recovery, and refutation. An arrow $A \to B$ means if we have an algorithm to solve $A$ for some $k$ value, then we get an algorithm for $B$ with the same $k$ value and comparable runtime.}
    \label{fig:relations}
\end{figure}
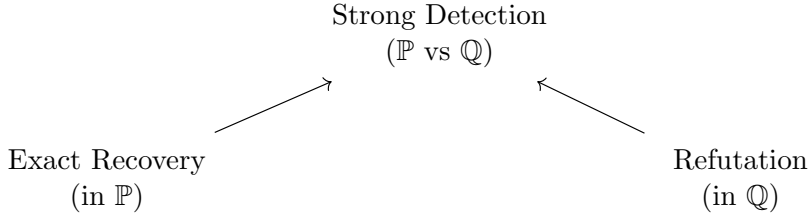

\begin{itemize}
    \item {\bf Recovery $\implies$ Detection}: Given an algorithm for exact recovery, we will give an algorithm for strong detection that works for the same value of $k$ and has comparable runtime. To construct our test, first try running the recovery algorithm on the given input $Y$. If it outputs $k$ vertices that actually form a $k$-clique in $Y$ then guess $\PP$, and otherwise guess $\QQ$. To prove correctness in the case $Y \sim \QQ$, it is important that a graph drawn from $\QQ$ has no $k$-clique (with high probability) so the recovery algorithm cannot possibly find one.
    
    We remark that for weaker notions of approximate recovery, it is less clear how to carry out this type of reduction (unless weak recovery can be automatically ``boosted'' to exact recovery, e.g.~\cite[Section~2.1]{ogp-sparse}). Some rather non-trivial reductions from detection to \emph{weak} recovery appear in~\cite{alg-contig,conj-sharp-sbm}. Also, connections between the \emph{information-theoretic} thresholds for detection and recovery appear in e.g.~\cite{BMVVX,rec-test-sbm}.
    
    \item {\bf Refutation $\implies$ Detection}: Given an algorithm for refutation, we will give an algorithm for strong detection that works for the same value of $k$ and has comparable runtime. To construct our test, first try running the refutation algorithm on the given input $Y$. If it outputs NO then guess $Y \sim \QQ$, and otherwise guess $Y \sim \PP$. To prove correctness in the case $Y \sim \PP$, it is important that a graph drawn from $\PP$ has a $k$-clique (with high probability) so the refuter is forced to output MAYBE.
\end{itemize}

\noindent The above arguments need only minimal assumptions on $\PP,\QQ$, namely that $\QQ$ does not contain the planted structure (with high probability) while $\PP$ does, and we need a way to verify that the recovery output is ``correct'' (e.g., actually forms a $k$-clique). Without these assumptions, there are simple but ``unnatural'' counterexamples to the above implications~\cite[Section~3.4]{BMVVX}.

In the planted clique setting, the three tasks (detection, recovery, refutation) are believed to be equally difficult,\footnote{In fact, planted clique has a recovery-to-detection reduction in the opposite direction~\cite{clique-s-to-d} (see also~\cite[Appendix~D]{clique-equiv}), although this is quite specific to the clique problem and not typical in other settings.} all with computational thresholds at $k \approx \sqrt{n}$~\cite{pcal,hopkins-thesis}, but this need not be the case in general. Some problems have a \emph{detection-recovery gap} where detection is strictly easier than recovery, maybe due to a ``trivial'' test such as counting the total edges in the graph; see e.g.~\cite{SW-estimation,det-rec-reduction}. A different choice of $\QQ$ might close this gap, but sometimes it seems difficult to close the gap completely if we want to keep $\QQ$ a ``simple'' distribution (see~\cite[Appendix~B.1]{SW-estimation}). Similarly, there are \emph{detection-refutation} gaps, e.g.~\cite{local-stats,det-rec-reduction}.

Recovery and refutation are incomparable because one pertains to $\PP$ while the other pertains to $\QQ$, so we cannot hope to formally relate these in general, at least without further assumptions about our choice of which $\PP$ and $\QQ$ to pair together. In some settings, a refutation algorithm implies a recovery algorithm, provided the refutation can be implemented in the sum-of-squares proof system (e.g.~\cite[Lemma~4.4]{tensor-pca-sos} for the case of tensor PCA). The planted $k$-coloring problem in $d$-regular graphs gives one example where the computational thresholds for recovery and refutation are different: for $k$ a large constant, strong detection and weak (non-trivial) recovery become easy once $d > k^2$ (the Kesten--Stigum threshold), while refutation becomes easy only once $d > 4k^2$ (see~\cite{spectral-planting}, especially Remark~2.7).

\subsection{Variations on the Low-Degree Framework}
\label{sec:variations-ld}

A notable strength of the low-degree framework is that it can be adapted to probe the computational complexity of different types of tasks like those discussed above. This is usually done in one of two ways. The first way is to directly phrase the question in terms of low-degree polynomials, defining a notion of ``success'' as we have done for detection and recovery in Section~\ref{sec:def-success}.

A second way to argue about hardness of some task is to relate it to some associated testing problem and then show low-degree hardness for the testing problem. I will refer to this as a ``two-stage'' argument, as it requires two components: (1) define a testing problem ($\PP$ versus $\QQ$), conjecture that it's hard (in terms of runtime), and justify the conjecture by proving a low-degree lower bound; (2) give a polynomial-time reduction showing that any hypothetical algorithm for the original task of interest would also yield an algorithm for the testing problem. As a result, we conclude hardness (in terms of runtime) for the task of interest, conditional on some unproven (albeit rigorously ``justified'') conjecture. While this makes for a slightly messy theorem statement, one major advantage of this approach is that testing lower bounds tend to be the easiest form of low-degree lower bounds to prove, particularly when $\QQ$ is a product measure.

We now discuss some tasks --- beyond detection and recovery --- where low-degree polynomials can be used to reason about computational complexity. Further details on the proof techniques are deferred to Section~\ref{sec:techniques}.

\begin{itemize}

\item {\bf Refutation}: We have already established in Section~\ref{sec:relations-tasks} that refutation is at least as hard as the associated detection problem, allowing us to argue for hardness of refutation by a two-stage argument as described above. In fact, there are many potential associated detection problems for a given refutation problem because we are free to choose any $\PP$ that contains the structure we're trying to refute. To prove the strongest hardness result, we need to choose $\PP$ that most effectively ``hides'' the planted structure. This may require a non-trivial construction for $\PP$, called a \emph{(computationally) quiet planting}.\footnote{This terminology comes from~\cite{sk-cert}, loosely inspired by~\cite{quiet}, but the idea existed earlier, e.g.\ for random $k$-SAT formulas~\cite{hiding-statmech}. See also~\cite{optimal-distinguishers-pc} for a sophisticated recent example involving planted clique.} Examples of two-stage arguments for refutation (or certification) include~\cite{rip-cert-clique,sk-cert,rip-cert,spectral-planting,non-neg-pca,cert-lift-mono}, and in some cases a different framework (not low-degree polynomials) is used to justify the intermediate conjecture.

We have seen above that hardness of refutation can be argued via connection to some related testing problem. However, this gives a result conditional on some conjecture, and it may be preferable to have an unconditional result. To this end, one can directly consider low-degree polynomials as refutation algorithms and define a notion of ``success'' analogous to the notion of separation (Definition~\ref{def:sep}). This allows for unconditional lower bounds against low-degree polynomials. We refer the reader to~\cite[Section~2.2]{coloring-clique} for the details. The methods above are formally incomparable to the more traditional refutation lower bounds based on convex programming hierarchies (see Section~\ref{sec:sos}).

\item {\bf Optimization}: Low-degree polynomial algorithms have also been considered for random optimization problems with no planted signal~\cite{ld-opt,indep-set,opt-ksat,barriers-perceptron,opt-balanced,opt-hyper,some-easy-ogp,strong-ld,sellke-spherical,strong-npp}. Here the precise definition of ``success'' is a bit more complicated and problem-specific, and we refer to those works for the details. The recent work~\cite{strong-ld} has improved many existing works by ruling out a weaker notion of success.

\item {\bf Fine-grained error}: Returning to detection and recovery, there are some finer-grained questions that we might hope to ask. For detection, we might care about the precise tradeoff between type I and II testing errors (i.e., false positives vs false negatives) in a regime where both converge to non-trivial constants. This is considered by~\cite{precise-error} using a two-stage argument, but we don't currently have unconditional results of this type for low-degree testing question. See also~\cite{optimal-distinguishers-pc} for some results of a similar flavor.

Similarly, for recovery we might be interested in pinning down the precise optimal mean-squared error in situations where this converges to some non-trivial constant. A few lower bounds of this type against \emph{constant-degree} polynomials are known~\cite{MW-amp,semerjian}, but extending this to higher degree remains open.

\item {\bf Query models}: We might also seek to ask more fine-grained questions about runtime, perhaps differentiating $O(n)$ from $O(n^2)$, and so on. This appears challenging because there does not seem to be a degree-runtime correspondence at this level of granularity. Some partial progress on this question can be made by considering non-adaptive query models as in~\cite{sublinear-clique,kiril-fourier}. These models can also be intrinsically motivated, as they capture situations where the statistician must choose upfront which data to collect. For instance, in the planted clique setting, this would mean deciding upfront which edges to observe (``query'') subject to a budget on the number of queries~\cite{sublinear-clique}. It remains largely open how to capture \emph{adaptive} queries within the low-degree framework, although some limited forms of adaptivity are handled by~\cite{quantum-ld} in the context of quantum learning.

\item {\bf Other examples}: Two-stage arguments can potentially be useful for tasks beyond the ones we have defined above. One example is the \emph{sampling} problem considered in~\cite{kunisky-sampling}, and some other examples are mentioned at the end of Section~\ref{sec:list}.

\end{itemize}

\section{Interpretation of Low-Degree Lower Bounds}
\label{sec:interp}

Why do we care to prove low-degree lower bounds, and how exactly should we interpret them? A conservative view is to think of them as simply ruling out some restricted class of algorithms, and as we'll see, this class captures various popular approaches. However, to fully justify why this framework is interesting, it would be remiss not to discuss its heuristic connection to time complexity, which was mentioned already in Section~\ref{sec:d-r-corr}. This brings us to the more ambitious view on low-degree lower bounds: they are a tool for making principled predictions and conjectures about time complexity in a world where we do not hope to resolve these questions definitively. In Section~\ref{sec:hope} we have already seen a few examples of settings where degree complexity appears to reliably track time complexity, and we will see in Section~\ref{sec:list} that many more such examples exist. On the other hand, we need to be careful about which settings the framework is appropriate for, especially as we start to deviate from the type of ``natural high-dim stat problems'' illustrated in Section~\ref{sec:high-dim}.

This section aims to explore the above issues in greater detail. Section~\ref{sec:eval-poly} discusses the complexity of evaluating polynomials, as one means to provide some justification for the heuristic degree-runtime correspondence from Section~\ref{sec:d-r-corr}. In Section~\ref{sec:alg-captured} we will see that many state-of-the-art algorithmic techniques can be captured by low-degree polynomials, which explains why degree complexity appears to track time complexity so reliably. In Section~\ref{sec:list} we will give a long list of ``success stories,'' that is, problems for which the degree-runtime correspondence appears to hold. Finally, Section~\ref{sec:conj} is devoted to the ``low-degree conjecture,'' that is, the ongoing effort to rigorously define the class of models where we expect the degree-runtime correspondence to hold. This includes a discussion of known ``counterexamples'' where the correspondence fails.

\subsection{Complexity of Evaluating Polynomials}
\label{sec:eval-poly}

A polynomial in $N$ variables of degree $D$ can have up to $\binom{N+D}{D}$ terms. This comes from counting the number of tuples $(k_1,\ldots,k_N) \in \ZZ_{\ge 0}^N$ with $k_1 + \cdots + k_N \le D$, using the ``stars and bars'' method. We will typically be interested in the case $D \le N^{1-\Omega(1)}$ and will typically have $N = n^{\Theta(1)}$ for our notion of problem size $n$, so the number of terms scales as $\binom{N+D}{D} = n^{\Theta(D)}$. Therefore, a naive term-by-term evaluation of the polynomial has runtime $n^{\Theta(D)}$, assuming the coefficients are easy to compute. (Here we treat addition and multiplication as atomic operations without worrying about bit complexity or numerical precision.) This justifies one direction of the degree-runtime correspondence in Hypotheses~\ref{hyp:poly-time} and~\ref{hyp:higher-D}, namely that time-$n^{O(D)}$ algorithms are at least as powerful as degree-$D$ polynomials, for any choice of $D = D_n$.

The other direction is more difficult to justify, and the following subsections will address this. One point we can make now is that the vast majority of degree-$D$ polynomials cannot be evaluated much faster than the naive term-by-term method. More formally, most degree-$D$ polynomials in $N$ variables require arithmetic circuits of size $\Omega(\sqrt{\binom{N+D}{D}})$; see~\cite[Ch~2]{circuit-lb-1} or~\cite[Ch~4]{circuit-lb-2}. So unless our polynomials of interest possess some kind of special structure, we should not hope to evaluate them faster than the naive way. On the other hand, we will see in Section~\ref{sec:alg-captured} below that some degree-$O(\log n)$ polynomials that naturally arise in high-dim stat problems actually do have special structure that allows them to be approximately evaluated in poly time rather than the naive runtime $n^{O(\log n)}$.

\subsection{Algorithms Captured by Polynomials}
\label{sec:alg-captured}

Recall that Hypothesis~\ref{hyp:poly-time} made the following informal assertion: if all degree-$O(\log n)$ polynomials fail to solve some statistical task, then so do all polynomial-time algorithms. This may feel like a bold claim, but remarkably it appears to hold up quite well for ``natural'' high-dim stat problems in the style of Section~\ref{sec:high-dim} (and we'll give a specific list of such problems in Section~\ref{sec:list}). One explanation for this is that a variety of powerful algorithmic techniques are captured by degree-$O(\log n)$ polynomials, and together these techniques tend to subsume the best known algorithms. This section is devoted to cataloging some examples of such algorithms and explaining why we consider logarithmic degree specifically. Most of the ideas here are guidelines rather than theorems, as it often requires some problem-specific work to transform these algorithms into rigorous low-degree upper bounds (in the sense of Section~\ref{sec:def-success}).

\begin{itemize}
    \item {\bf Algorithms From Polynomials}: The first example is an ``obvious'' one, but it is worth mentioning that various algorithmic results have been directly obtained from low-degree polynomials. For instance, this includes performing inference on random graphs by counting the number of occurrences of certain subgraphs, as in (for example)~\cite{log-density,MNS-proof,HS-bayesian,graph-match-nearly,graph-match-trees,graph-match-otter,decorated-trees}. In some of these cases~\cite{HS-bayesian,graph-match-trees,graph-match-otter,decorated-trees}, degree-$\Theta(\log n)$ polynomials are used, and the ``color coding'' trick~\cite{color-coding} is employed to approximately evaluate these in polynomial time.
        
    \item {\bf Spectral Methods}: This refers to a general class of approaches where one constructs some matrix using the input and then uses the leading eigenvalue(s) or eigenvector(s) to perform inference, an early example being~\cite{alon-clique} for the planted clique problem. In general, one can imagine building a (usually symmetric) matrix $M = M(Y)$ where each entry $M_{ij}$ is a constant-degree polynomial in the input variables $Y$. This might be the ``obvious'' choice of matrix, such as the adjacency matrix if the input is a graph, but various state-of-the-art results require a non-trivial choice of matrix: the non-backtracking operator~\cite{spectral-redemption,nb-spectrum} or Bethe Hessian~\cite{bethe-hessian} for community detection, and other examples~\cite{fast-sos,kikuchi,spectral-phase-retrieval,bp-unstable}, just to name a few. Remark~\ref{rem:log-deg} below gives more details on how to implement spectral methods using degree-$O(\log n)$ polynomials, using power iteration and related ideas. We will see that the matrix must be ``well-conditioned'' for this to work, and in fact, an example of~\cite{conj-false} shows that low-degree polynomials cannot compute leading eigenvalues in general. For rigorous results that use low-degree lower bounds to prove failure of spectral methods (appropriately defined), see~\cite[Section~4.2.3]{ld-notes} and~\cite[Section~4.2]{planted-vec}.
    
    \item {\bf Approximate Message Passing (AMP)}: The AMP framework~\cite{amp,BM-amp} (see \cite{amp-tut-unif,amp-tut-concise} for a survey) has been widely successful at producing state-of-the-art iterative algorithms for a variety of high-dim stat problems. These algorithms often give the best known performance among efficient algorithms in a sharp sense, such as the precise limiting value of mean squared error. AMP algorithms have been used extensively for planted problems (e.g.~\cite{FR-amp,DM-sparse-pca,MMSE,submatrix-amp,glm-amp}), and more recently for non-planted optimization problems (e.g.~\cite{opt-sk,opt-spin-glass}). We generally expect that AMP-style algorithms can be effectively approximated by low-degree polynomials, and this approximation has been rigorously analyzed in some settings~\cite{MW-amp,sos-amp-robustly,fourier-iterative,alg-universality}. These results show that $O(1)$ iterations of AMP can be captured by degree-$O(1)$ polynomials, but some settings require $O(\log n)$ iterations or a spectral initialization~\cite{amp-spectral,amp-spectral-glm}. It is a good open question to prove that degree-$O(\log n)$ polynomials capture AMP in these settings, perhaps using $O(\log n)$ rounds of power iteration followed by $O(1)$ iterations of AMP.

    \item {\bf Local Algorithms on Sparse Graphs}: Consider a non-planted optimization problem where the input is a sparse random graph $G(n,p/n)$ for a constant $p > 0$. Roughly speaking, a \emph{local algorithm} (see~\cite{GS-ogp}) is one where the output at each vertex depends only on its constant-radius neighborhood in the graph. A simple example for constructing an independent set is to assign an independent $\cN(0,1)$ label to each vertex, and then include every vertex that has a larger label than all its neighbors. It is shown in~\cite[Remark~3.2]{indep-set} that any local algorithm can be well approximated by a degree-$O(1)$ polynomial.    
\end{itemize}

\begin{remark}[Why Logarithmic Degree?]
\label{rem:log-deg}
To illustrate where the choice of logarithmic degree arises, let's return to the important example of spectral methods from above and see how to implement these as polynomials. Suppose $M = M(Y)$ is an $n \times n$ symmetric matrix whose entries are degree-$O(1)$ polynomials in $Y$. A scenario to have in mind is the following: under $\QQ$, all eigenvalues of $M$ are (with high probability) confined to some compact interval, say $[-1,1]$; under $\PP$, most eigenvalues are again confined to $[-1,1]$ but there is a single ``signal'' eigenvalue at $1+\epsilon$ for a constant $\epsilon > 0$, and the associated leading eigenvector has non-trivial correlation with the hidden vector we aim to estimate. This situation arises in spiked Wigner or Wishart matrix models, for instance~\cite{BBP,BS-spiked,FP-wigner,CDF-wigner,BN-eigenvec}. The leading eigenvalue solves detection, and this can be approximated by the polynomial
\[ f(Y) := \Tr(M^{2k}) = \sum_{i=1}^n \lambda_i^{2k} \]
for an integer $k > 0$, where $\{\lambda_i\}$ are the eigenvalues of $M$. For large enough $k$, the largest term in the sum will dominate, so $f(Y) \approx \max_i \lambda_i^{2k}$. The following heuristic calculation shows that $k = O(\log n)$ suffices for successful detection: under $\QQ$ we have $f(Y) \le n \cdot 1^{2k} = n$ and under $\PP$ we have $f(Y) \ge (1+\epsilon)^{2k}$, so we need $(1+\epsilon)^{2k} > n$, i.e., $k > (\log n)/[2 \log(1+\epsilon)]$.

For recovery, we instead wish to estimate the leading eigenvector, which can be done by power iteration: $M^k u$ for some initial vector $u$, say the all-ones vector. Provided we have a $(1+\eps)$-factor ``spectral gap'' between the two largest (in magnitude) eigenvalues as above, a standard analysis (expand $u$ and $M^k u$ in the eigenbasis of $M$) shows that again $k = O(\frac{\log n}{\epsilon})$ suffices to reach a vector that is well-correlated with the leading eigenvector. If we aim to turn this into a rigorous low-degree upper bound, i.e., to show $\MMSE_{\le D}$ (Definition~\ref{def:mmse}) is small, this should be formulated in terms of matrix MSE: we aim to estimate $vv^\top$ where $v$ is the leading eigenvector, to avoid the ambiguity between $v$ and $-v$. For this we can use the estimator $M^k$ times a suitable scalar multiple. The $(i,j)$ entry of $M^k$ can be viewed as a sum over length-$k$ walks from vertex $i$ to vertex $j$ in the complete graph, and for ease of analysis, it may be advantageous to restrict to \emph{self-avoiding} walks. See~\cite[Section~2.6]{HS-bayesian}, where the details are carried out for the spiked Wigner model to yield a low-degree upper bound at degree $O(\log n)$ above the BBP threshold. Degree $\Omega(\log n)$ is necessary here~\cite[Thm~2.6]{sharp-est}.
\end{remark}

In light of the above, one perspective on low-degree lower bounds is that they can be used to rule out many concrete algorithm classes of interest, either heuristically or rigorously. A good direction for future work is to expand the class of algorithms that can be rigorously captured by low-degree polynomials. Currently, most low-degree upper bounds are proved on a problem-by-problem basis, but perhaps more general tools can be developed.

\subsection{Success Stories}
\label{sec:list}

Perhaps one of the main reasons to believe that low-degree polynomials are useful for predicting average-case complexity is simply the large amount of ``empirical'' evidence we have gathered by studying many different problems. The goal of this section is to list many models where the low-degree framework has been successful. To be considered a ``success story,'' a model should have a stat-comp gap and the degree complexity should track the (best known) time complexity as described in Section~\ref{sec:d-r-corr}, similar to the examples in Section~\ref{sec:hope}.

An ideal success story would have matching upper and lower bounds, showing that low-degree polynomials succeed in the ``easy'' regime where poly-time algorithms are known, but fail in the ``hard'' regime. However, the list will also include some settings where the results are less complete --- e.g., upper bounds are missing or lower bounds are suboptimal --- as long as the low-degree results appear consistent with the known time complexity. The list certainly will not contain examples where degree-$O(\log n)$ polynomials are provably ``beaten'' by some poly-time algorithm; these will be discussed in Section~\ref{sec:counter}.

We now begin the list, focusing for now on detection, recovery, and optimization. Many of the models listed here are extensively studied, and for brevity we only include here the citations that carry out the low-degree analysis.

\paragraph{List of low-degree ``success stories'':}

\begin{itemize}

\item {\bf Planted problems in random graphs}: Planted clique~\cite{hopkins-thesis,SW-estimation}, including ``secret leakage'' variants~\cite{secret-leakage}, ``masked'' variants~\cite{sublinear-clique}, and multi-clique variants~\cite{coloring-clique}; planted dense subgraph~\cite{SW-estimation,subhypergraph,sharp-est} and multi-community variants~\cite{planted-v-planted}; community detection in the stochastic block model~\cite{HS-bayesian,spectral-planting,graphon,kunisky-coordinate-2,sharp-est,sbm-many}, including ``mixed-membership''~\cite{HS-bayesian}, ``degree-correlated''~\cite{small-community}, and ``multi-layer''~\cite{multi-layer-sbm} variants; detecting correlated random graphs (graph matching)~\cite{graph-match-trees,correlated-er,weighted-graph-matching} or correlated block models~\cite{correlated-sbm}; planted dense cycles~\cite{dense-cycles}; detecting geometry~\cite{kiril-fourier} or algebraic structure~\cite{algebraic-graphs} in random graphs; planted ranked community in a directed graph~\cite{ranked}; planted random subgraph~\cite{crypto-random-subgraph}; arbitrary structures planted in random graphs~\cite{huleihel-general-hidden,counting-stars,EH-detecting-arbitrary}.

\item {\bf Planted problems in random matrices}: Sparse PCA~\cite{sos-detect,subexp-sparse,sparse-adversarial}; sparse clustering~\cite{sparse-clustering}; other spiked Wigner and Wishart models~\cite{HS-bayesian,sk-cert,ld-notes,spectral-planting,MW-amp,improper-regression,sharp-est}, including negatively spiked Wishart~\cite{sk-cert,rip-cert,improper-regression} and variants with general noise channels~\cite{kunisky-coordinate}; planted submatrix~\cite{SW-estimation,sharp-est} and variants with multiple communities~\cite{planted-v-planted,hidden-submatrices}; matrix denoising~\cite{semerjian}; spiked cumulant model~\cite{spiked-cumulant}.

\item {\bf Planted problems in random tensors \& hypergraphs}: Tensor PCA~\cite{sos-detect,ld-notes,tensor-cumulants} and sparse variant~\cite{sparse-tensor}; tensor decomposition~\cite{ld-tensor-decomp}; tensor-on-tensor association detection~\cite{t-on-t-association}; scalar-on-tensor regression~\cite{t-on-t-regression}; planted clique in hypergraphs~\cite{LZ-tensor}; planted dense subhypergraph~\cite{subhypergraph}; planted hyperloops~\cite{crypto-hyperloops}.

\item {\bf Other planted problems}: Sparse linear regression~\cite{fp} and ``mixed'' variant~\cite{mixed-sparse-reg}; group testing~\cite{grp-testing}; sparse CCA~\cite{sparse-cca}; planted vector in a random subspace, or equivalently, non-gaussian component analysis~\cite{sparse-adversarial,planted-vec}; group synchronization~\cite{group-synch-ld,kunisky-coordinate-2}; single-spike block mixtures~\cite{heavy-tailed}; gaussian mixture models~\cite{sq-ld,low-rank-gmm,lloyd-gmm,hd-clustering}; gaussian graphical models~\cite{sq-ld}; learning truncated gaussians~\cite{truncated-gaussian}; single-index models~\cite{single-index}; shuffled linear regresion~\cite{shuffled-regression}.

\item {\bf Non-planted optimization problems}: Maximum independent set in sparse graphs~\cite{ld-opt,indep-set,strong-ld} and hypergraphs~\cite{opt-hyper}, including the ``balanced'' variant in bipartite graphs \cite{opt-balanced}; maximum clique in $G(n,1/2)$~\cite{indep-set,strong-ld}; $k$-SAT~\cite{opt-ksat,strong-ld}; spin glass optimization problems~\cite{ld-opt,strong-ld,sellke-spherical}; perceptron models~\cite{barriers-perceptron,strong-ld}.

\end{itemize}

I would consider this a resounding success, in that low-degree polynomials give a unifying explanation for the apparent computational barriers in many different statistical problems, and therefore provide a principled way to predict where the barriers lie in new problems, at least for problems that are similar in style to the examples above.

In some sense, it is the low-degree \emph{upper bounds} that may be most relevant to the point we hope to make here, as they amass evidence for the following phenomenon: \emph{whenever poly-time algorithms succeed (e.g., at strong detection) there should exist a degree-$O(\log n)$ polynomial that succeeds (in the corresponding sense, e.g., strong separation)}. The contrapositive of this statement gives meaning to the lower bounds. Many of the above citations explicitly include low-degree upper bounds for the same notions of success (see Section~\ref{sec:def-success}) used in the lower bounds: \cite{SW-estimation,indep-set,opt-ksat,grp-testing,ld-tensor-decomp,MW-amp,planted-v-planted,dense-cycles,coloring-clique,subhypergraph,graphon,sublinear-clique,semerjian,ranked,sharp-est,sbm-many}. Various other algorithmic results appear to implicitly contain low-degree upper bounds, e.g.~\cite{HS-bayesian,graph-match-nearly,heavy-tailed-saw,graph-match-trees,graph-match-otter,decorated-trees}.

Most of the works cited above are concerned with predicting the limits of poly-time algorithms, but we would also like to have empirical evidence for the degree-runtime correspondence for \emph{super-polynomial} runtimes from Hypothesis~\ref{hyp:higher-D}. There are relatively fewer success stories for this so far, but also no known counterexamples (beyond the ones applicable to poly time, discussed in Section~\ref{sec:counter}). The most interesting examples to look at here are those with a smooth runtime-SNR tradeoff (see Remark~\ref{rem:sharp-smooth}), and the success stories of this type include variants of sparse PCA~\cite{subexp-sparse,sparse-adversarial}, as well as (sparse) tensor PCA~\cite{ld-notes,sparse-tensor}, and planted submatrix/subgraph~\cite{sharp-est}. There are also examples where the best known algorithms in the hard regime need fully exponential time $\exp(n^{1-o(1)})$, and we have degree-$\tilde\Theta(n)$ lower bounds to match; these include spiked matrix models~\cite{sk-cert,ld-notes,spectral-planting}, the stochastic block model~\cite{spectral-planting}, and various random optimization problems~\cite{strong-ld}. However, there is a notable gap in our knowledge here: I am not aware of any low-degree \emph{upper bounds}\footnote{Here I am being slightly picky about what constitutes an upper bound, namely it should establish strong/weak separation or bound $\MMSE_{\le D}$ as defined in Section~\ref{sec:def-success} and justified in Sections~\ref{sec:comm-sep}--\ref{sec:comm-upper}.} at super-logarithmic degree to match the algorithms of super-polynomial runtime mentioned above. It would be comforting to confirm that these exist as expected, to make our lower bounds more meaningful.

For completeness, we will mention a few more success stories where low-degree lower bounds have been used to draw useful conclusions \emph{beyond} the settings of detection/recovery/optimization, usually via connection to some testing problem as in Section~\ref{sec:variations-ld}. First we have certification problems: certifying bounds on various constrained PCA problems~\cite{sk-cert,spectral-planting,non-neg-pca}, certifying the restricted isometry property~\cite{rip-cert}, refuting $k$-colorability in $G(n,1/2)$~\cite{coloring-clique}, and certifying well-spreadness of random matrices~\cite{well-spread}. Also, low-degree lower bounds have been used to deduce conclusions about various non-Bayesian problems such as heavy-tailed statistics~\cite{heavy-tailed}, robust mean testing~\cite{robust-mean-testing}, sampling general Ising models~\cite{kunisky-sampling}, and improperly learning sparse regression~\cite{improper-regression}. Finally, the low-degree framework has been adapted to quantum learning problems~\cite{quantum-ld}.

\subsection{The Low-Degree Conjecture}
\label{sec:conj}

When a new problem comes along and we prove low-degree hardness for it, how confident can we be that there's actually no poly-time algorithm? We've seen above that low-degree polynomials have a good track record of predicting hardness for a certain style of problems, but how can we tell if our new problem belongs in that same class? This is especially important for high-stakes applications such as cryptography, where recent work has used low-degree lower bounds to justify the security of certain cryptosystems~\cite{crypto-graph,crypto-hyperloops,crypto-random-subgraph}.

We now discuss a line of work that seeks to formally define the class of ``natural high-dim stat problems'' for which degree-$O(\log n)$ polynomials are expected to be at least as powerful as all poly-time algorithms. This allows us to state precise conjectures, which can be ``stress tested'' and refined over time. We don't realistically hope to prove such a conjecture --- as this would imply P $\ne$ NP --- but the goal here is to put our beliefs about the low-degree framework into precise, falsifiable, terms.

\subsubsection{The Conjecture of Hopkins}

We use the term ``low-degree conjecture'' loosely to describe the belief that for a certain class of problems, if low-degree polynomials fail then so do all poly-time algorithms. A precise form of the low-degree conjecture was stated in the PhD thesis of Hopkins~\cite[Conjecture~2.2.4]{hopkins-thesis}, inspired by~\cite{pcal,sos-detect,HS-bayesian}. Although this conjecture was recently refuted~\cite{conj-false}, we will spend some time discussing it. We refer the reader to the original source for the full details, but the conjecture essentially states that for a certain class of testing problems $\PP = \PP_n$ versus $\QQ = \QQ_n$, if the ratio
\[ \sup_{f \in \RR[Y]_{\le D}} \frac{\EE_{Y \sim \PP}[f(Y)]}{\sqrt{\EE_{Y \sim \QQ}[f(Y)^2]}} \]
is $O(1)$ as $n \to \infty$ for some $D \ge (\log n)^{1+\epsilon}$ with $\epsilon > 0$ fixed, then no poly-time algorithm can achieve strong detection (as defined in Section~\ref{sec:task-def}). The reader might recognize this quantity as the ``advantage'' from Eq.~\eqref{eq:adv}, and boundedness of this quantity implies that no degree-$D$ polynomial can strongly separate $\PP$ and $\QQ$ (Proposition~\ref{prop:adv-to-sep}). We have discussed previously (Section~\ref{sec:alg-captured}) how various important algorithms are captured by degree-$O(\log n)$ polynomials, so the choice of $(\log n)^{1+\epsilon}$ is just slightly larger than this, to be safe. The crux of the conjecture is the assumptions imposed on $\PP,\QQ$, which we now summarize.
\begin{itemize}
    \item There are $\binom{n}{m}$ observed variables for a fixed integer $m$. These can be thought of as noisy $m$-wise observations among $n$ unknown variables. For instance, in planted clique, $n$ would be the number of vertices, and $m=2$.
    \item Under the null model $\QQ$, the observed variables are i.i.d.
    \item The distribution $\PP$ is invariant under permutation of the $n$ underlying variables.
    \item The testing problem must be robust to a small amount of noise. Formally, the conjecture states that if $\PP,\QQ$ have bounded advantage then no poly-time algorithm can distinguish $\QQ$ from a \emph{noisy} version of $\PP$, where a small constant fraction of entries are resampled according to the distribution $\QQ$.
\end{itemize}

\begin{remark}[Coordinate Degree]
\label{rem:coordinate-deg}
\noindent The conjecture of Hopkins also differs from our treatment in that the advantage is not defined using $\RR[Y]_{\le D}$ (the polynomials of degree at most $D$) but rather the polynomials of \emph{coordinate degree} at most $D$. This means each monomial involves at most $D$ variables, but can have arbitrary degree in those variables. For binary input variables, this coincides with the usual notion of degree. For the ``natural'' high-dim stat problems I'm aware of, degree and coordinate degree seem to behave roughly the same, although see~\cite[Section~2.4]{morris} and~\cite{strong-npp} for some possible counterpoints. Lower bounds against coordinate degree are formally stronger. On the other hand, degree may be more natural in some contexts due to its basis invariance, and may yield convenient formulas such as Eq.~\eqref{eq:gaussian-adv} in Section~\ref{sec:together}. Some proof techniques specifically leverage coordinate degree~\cite{KM-tree,ld-bot,kunisky-coordinate,kunisky-coordinate-2,ld-bot-opt}.
\end{remark}

\subsubsection{Known Counterexamples}
\label{sec:counter}

The conjecture of Hopkins stood for about 7 years --- aside from a refinement of~\cite{HW-counter} mentioned below --- and then very recently was refuted by~\cite{conj-false}. How to ``fix'' the conjecture is a subject of ongoing work and I will not attempt to propose an answer here. Instead, we will discuss which features of the conjecture are known to be necessary by detailing some known ``counterexamples,'' i.e., settings where low-degree polynomials are beaten by some poly-time algorithm.

\begin{itemize}

    \item {\bf XOR-SAT}: Suppose there are binary variables $x_1,\ldots,x_n \in \{0,1\}$ and we are given $m$ equations of the form $x_i + x_j + x_k \equiv b \pmod{2}$, with $i,j,k$ chosen randomly from $[n]$. We aim to distinguish between two cases: ($\QQ$) each $b$ is chosen randomly from $\{0,1\}$ and ($\PP$) the $b$ values are all chosen to be consistent with some planted assignment $x$. In the regime $m \gg n$, there is no satisfying assignment under $\QQ$ and so $\PP,\QQ$ can be distinguished in poly time by using Gaussian elimination to solve the system of linear equations over the finite field $\FF_2$. However, for $m \ll n^{3/2}$, low-degree polynomials fail to distinguish $\PP,\QQ$. This has long been documented as an example where popular ``predictors'' of average-case hardness are ``fooled,'' including the statistical query model~\cite{sq-csp}, sum-of-squares~\cite[Section~3.1]{boaz-notes}, and statistical physics~\cite[Section~IV.B]{ZK-phys-survey}, as well as low-degree polynomials~\cite[Section~2.2]{hopkins-thesis}. All of these sources argue, however, that this is not a ``serious'' counterexample because Gaussian elimination is a ``brittle'' algebraic algorithm that breaks down when a small amount of noise is added: if we consider a modified $\tilde{\PP}$ where only $1-\epsilon$ fraction of constraints are satisfied for an arbitrary constant $\epsilon > 0$, the problem of distinguishing $\QQ,\tilde{\PP}$ is actually believed to be hard when $m \ll n^{3/2}$. The threshold $m \approx n^{3/2}$ predicted by low-degree polynomials (and other methods) represents a ``robust'' computational threshold, that is (depending on the context) arguably more meaningful than the ``true'' computational threshold! For this reason, the assumption of noise-robustness in Hopkins' conjecture is essential.

    \item {\bf LLL}: The Lenstra--Lenstra--Lov\'asz (LLL) algorithm for \emph{lattice basis reduction}~\cite{LLL} has been used to ``break'' a number of suspected computational barriers in certain statistical problems~\cite{corr-ret,lll-reg,lll-neuron,lll-cluster,lll-ngca}, building on a solution to the average-case \emph{subset sum} problem~\cite{LO-lll,F-lll}. To give one example, for the task of finding a planted $\{\pm 1\}$-valued vector in a random $d$-dimensional subspace of $\RR^n$, low-degree algorithms require $d \ll \sqrt{n}$~\cite{planted-vec} but LLL can succeed for all $d \le n-1$~\cite{lll-cluster,lll-ngca}. This result, and the others above, crucially require the problem to be ``noiseless'' in some sense, since LLL (like Gaussian elimination) has very low noise tolerance. As a result, these examples do not refute Hopkins' conjecture, but they do carry an important warning about noiseless problems. One point made in~\cite[Section~1.3]{lll-cluster} is that planted clique is noiseless in the sense that it does not satisfy Hopkins' notion of noise-robustness.

    \item {\bf Error-Correcting Codes}: In~\cite{HW-counter}, a few examples based on error-correcting codes illustrate two different points. The first example shows how to refute Hopkins' conjecture with the symmetry assumption ($\PP$ is permutation-invariant) removed. The second example shows that the original form of Hopkins' conjecture is false when $\QQ$ has entries i.i.d.\ from a \emph{continuous} distribution, but the conjecture can be ``fixed'' by changing the type of noise added to $\PP$ in the robustness assumption. However, even the fixed version has now been refuted by~\cite{conj-false}, which also gives a construction based on a (highly symmetric) error-correcting code.

    \item {\bf Heavy-Tailed Noise}: The ``classical'' high-dim stat problems tend to involve distributions with well-behaved moments, e.g., subgaussian, hypercontractive, or similar. The case of heavy-tailed noise, where the moments grow very quickly, is less well-explored. In the extreme case where some moments are infinite, polynomials cease to have finite mean and variance. Based on this idea,~\cite[Section~2.4]{morris} constructs a testing problem where low-degree polynomials are inferior to some other poly-time algorithm. This does not refute Hopkins' conjecture because $\QQ$ is not i.i.d.

    \item {\bf Small Spectral Gap}: In Remark~\ref{rem:log-deg} we have argued that low-degree polynomials capture spectral methods, as long as the leading eigenvalue under $\PP$ exceeds that under $\QQ$ by a $1+\eps$ factor for a constant $\eps > 0$. It is less clear what happens when $\eps = o(1)$, which would occur, for instance, in the spiked Wigner model with SNR slightly above the critical threshold but approaching the threshold at some rate. This is a scenario where we may expect to see spectral methods succeed but low-degree polynomials fail. However, this will not yield a counterexample to Hopkins' conjecture, since the extra noise added to $\PP$ effectively decreases the SNR by a small constant, moving us below the threshold. An example of~\cite{conj-false} gives a setting (with non-product $\QQ$) where low-degree tests fail yet the leading eigenvalue works, even in the presence of noise.
    
    \item {\bf Broadcasting on Trees}: The \emph{broadcasting on trees} problem is different in structure from the problems considered in Hopkins' conjecture. A random process propagates down the vertices of a tree, and the goal is to infer the label of the root node given the labels of the leaves (so this is a recovery rather than detection problem). Information-theoretically optimal performance is given by the computationally efficient \emph{belief propagation (BP)} algorithm, but low-degree polynomials fail unless the model parameters lie above the so-called \emph{Kesten--Stigum (KS)} threshold~\cite{KM-tree,ld-bot,ld-bot-opt}. As discussed in~\cite{ld-bot}, the KS threshold is known to be the fundamental limit for certain complexity measures, including some notion of robust reconstruction, so the low-degree threshold still seems to carry some meaning here.
    
\end{itemize}

\subsubsection{Potential Extensions and Outlook}

Hopkins' conjecture --- even if it were true --- applies only to a limited fraction of the settings where the low-degree framework has been fruitful, and so in principle there might be room to expand or generalize the conjecture (subject to appropriate assumptions to make it true). The original conjecture considers the case of ruling out strong detection in testing between some ``planted'' distribution $\PP$ and some i.i.d.\ ``null'' $\QQ$. In the time since that conjecture was formed, we have started using low-degree polynomials to make predictions about other types of statistical tasks, including weak detection, as well as recovery, optimization, and refutation. We have also found that low-degree polynomials appear to make reliable predictions for more general testing problems than those treated in the conjecture, including certain settings where $\QQ$ is not a product measure (see Section~\ref{sec:planted-v-planted-pf}) or models like sparse PCA (as defined in Section~\ref{sec:mot-ex}) that don't quite fit the symmetry assumptions.

So far, there have not been attempts to formulate a precise ``low-degree conjecture'' for other tasks like recovery, optimization, or refutation (although low-degree lower bounds do exist in these settings; see Section~\ref{sec:variations-ld}). For the case of non-planted optimization problems, there is some discussion in Appendix~A of~\cite{ld-opt} surrounding what we do (and don't) expect low-degree polynomials to achieve (e.g., we should only ask them to find an approximate solution rather than \emph{the} global optimum). Also,~\cite{strong-npp} shows that the low-degree heuristic does not hold up in a certain optimization problem --- \emph{number partitioning} --- and suggests that this might be fixable by considering coordinate degree instead. As a general rule, when we are attempting to use the low-degree framework in a setting whose ``style'' deviates somewhat from existing work, it becomes especially important to prove low-degree \emph{upper bounds} (success in the ``easy'' regime) to illustrate that the notion of ``success'' used in our lower bounds is meaningful.

Stronger versions of the low-degree conjecture appear in~\cite{precise-error,alg-contig,conj-sharp-sbm}, dealing with the precise limiting value of the advantage rather than just its boundedness. However, these works do not attempt to specify a precise class of models where the conjecture is expected to apply.

Where do we go from here? The numerous success stories suggest that the ``low-degree conjecture'' holds true for a broad class of models, but describing this class precisely has proven difficult. Part of the difficulty is that it is hard to know whether one has succeeded without waiting years to see if the conjecture survives the test of time. On the plus side, the counterexamples we have discovered along the way help to guide our interpretation of low-degree lower bounds. Degree complexity often appears to carry some intrinsic meaning even in cases where it doesn't quite align with time complexity. As we continue to discover both success stories and counterexamples, we better our understanding of how to form responsible conjectures about time complexity, even in the absence of an all-purpose conjecture.

\section{Relations to Other Frameworks}
\label{sec:relations}

Besides low-degree polynomials, there are a variety of other popular frameworks for probing the average-case complexity of statistical problems. Each brings a different perspective to the table, and this section will attempt to elucidate some of the, often subtle, relations among them. Most of the frameworks (with the exception of reductions) are tied to a particular class of algorithms that can be ruled out as a form of evidence for hardness. The proponents of each framework tend to believe that the associated class of algorithms is ``all-powerful,'' in the sense that if they fail at some task (within some nebulously-defined class of ``natural'' tasks) then so do all efficient algorithms.

First, what do we desire from one of these frameworks? There are a few possible answers:

\begin{itemize}

    \item {\bf Reliable Predictions}: Ideally, a framework for average-case complexity should make reliable predictions about when problems are hard, consistent with our best known algorithmic results. Examples like XOR-SAT from Section~\ref{sec:counter} will basically ``fool'' all these frameworks, so we need to set our expectations accordingly. For a framework to be considered ``reliable'' I will ask that it at least predicts the ``right'' thresholds for problems like planted clique, tensor PCA, and the other examples mentioned in Section~\ref{sec:high-dim}. This criterion suggests a ``competitive'' viewpoint where we pit the different frameworks against each other to see which makes the best predictions.
    
    \item {\bf Failure of Algorithms}: Another desirable outcome is to prove unconditionally that various specific classes of algorithms fail to solve the (conjectured) ``hard'' problems. In this sense, the different frameworks tend to complement each other by ruling out different types of algorithms. This suggests a more ``harmonious'' viewpoint, where we aim to build rigorous connections between the different frameworks in order to rule out many different algorithms all at once. Some ``meta-theorems'' have emerged that allow us to automatically deduce hardness in one framework from hardness in another, with a prime example being the connection between low-degree polynomials and the statistical query model~\cite{sq-ld}.
    
\end{itemize}

In the discussion of each framework below, we will cover both the ``competitive'' view (judging the quality of predictions) and the ``harmonious'' view (exploring known connections to other frameworks and algorithms). At the end, Section~\ref{sec:summary} provides a summary.

\subsection{Reductions}
\label{sec:reductions}

Average-case reductions establish rigorous connections between different average-case problems, allowing us to argue that \emph{if} one problem is hard, then so is another. This is different in spirit from the other frameworks we will discuss, since the point here is not to study a restricted class of algorithms but to uncover fundamental relationships \emph{between} problems, which is a separate (but related) goal of independent interest.

It has become popular to prove hardness for statistical problems \emph{conditional} on the conjectured hardness of the planted clique (PC) problem for $k \ll \sqrt{n}$. The work of~\cite{BR-reduction} was influential in first bringing this idea to the statistics community (although PC was used earlier as a hardness assumption in other contexts, e.g.~\cite{hiding-cliques-crypto}), and subsequent work has shown PC-hardness for a variety of statistical problems, e.g.~\cite{ma-wu-reduction,HWX-reduction,BBH-reduction,BB-opt} and references therein. PC-hardness is perhaps the ``gold standard'' in average-case complexity, since we need only believe hardness for the PC problem and can avoid the question ``which class of problems is suitable?''\ that we have labored over in Section~\ref{sec:conj}.

Hardness results based on the standard PC assumption are somewhat limited to problems of similar structure to PC. The class of problems has been expanded by starting from other assumptions such as the \emph{planted dense subgraph recovery conjecture}~\cite{BBH-reduction} or \emph{secret leakage planted clique conjecture}~\cite{secret-leakage}, and this required new tools such as \emph{dense Bernoulli rotations}~\cite{secret-leakage}. The low-degree framework complements the reduction framework by backing up these conjectures, as in~\cite{SW-estimation} and~\cite{secret-leakage} respectively. A decent portion of the problems listed in Section~\ref{sec:list} now have some form of reduction-based hardness~\cite{ma-wu-reduction,HWX-reduction,BBH-reduction,univ-lb,BB-opt,secret-leakage}. Other recent advances in our machinery for average-case reductions include~\cite{extra-triangles,det-rec-reduction,reduction-some-stat,equiv-reduction}.

Another line of reduction-based results start by assuming hardness of worst-case lattice problems --- a standard cryptographic assumption --- and use this to establish hardness for certain statistical problems~\cite{clwe,lll-neuron,lwe-mixtures,crypto-massart,crypto-agno,sparse-reg-lattice,improved-hardness-halfspaces,crypto-score,lwe-linear}. Notably, lattice-based hardness has been shown for the \emph{symmetric binary perceptron}~\cite{VV-perceptron}, a random optimization problem which differs from the other examples above by being \emph{non-planted}.

There are a number of synergies between the low-degree and reduction approaches. As mentioned above, the low-degree framework can provide evidence to support the starting conjectures used for reductions. It can also be helpful to have a low-degree prediction for where the threshold is expected to lie, before attempting a reduction. Some phenomena predicted by low-degree polynomials --- such as sharp thresholds (see Remark~\ref{rem:sharp-smooth}) --- have not yet been recovered using reductions, and this remains a good question for future work. Recent work of~\cite{optimal-distinguishers-pc} transfers low-degree hardness to reduction-based hardness in the context of planted clique, resulting in a proof that low-degree distinguishers for planted clique are optimal in a sharp sense, conditional on the PC conjecture.

\subsection{Sum-of-Squares}
\label{sec:sos}

The \emph{sum-of-squares (SoS) hierarchy} is an algorithmic tool based on semidefinite programming (SDP). It gives increasingly powerful algorithms --- parameterized by $d$, the degree of polynomials used --- at the expense of increased runtime $n^{O(d)}$. Its relation to the low-degree framework is multifaceted, and the histories of these two approaches are intertwined. We refer the reader to~\cite{sos-survey} for a survey on the use of SoS in high-dim stats. The discussion below will center around the following three points:
\begin{enumerate}
    \item[(1)] SoS is generally viewed as a ``stronger'' class of algorithms than low-degree polynomials.
    \item[(2)] But for high-dim stat problems, it is conjectured that SoS is not \emph{strictly} stronger.
    \item[(3)] Furthermore, SoS lower bounds often pertain to the \emph{refutation} task, making them (in a sense) \emph{weaker} than the associated low-degree lower bounds for \emph{detection}.
\end{enumerate}

Regarding (1), algorithms based on SoS have proved useful well beyond Bayesian inference problems, including settings where the input has worst-case or adversarial components, e.g.~\cite{mixture-robust-sos,robust-moment-est,heavy-tailed,robust-sparse-mean,semi-random-clique}. In contrast, low-degree polynomials primarily seem useful in fully Bayesian settings, so in this sense, SoS represents a more powerful class of algorithms. SoS has been championed as a ``meta-algorithm'' that potentially achieves the best possible guarantees for a wide range of computational problems~\cite{sos-quest}. Some of the initial excitement about SoS was fueled by its ability to solve certain candidate hard instances for the (worst-case) Unique Games problem, where weaker hierarchies failed~\cite{sos-hyp}. To give one recent illustration of the power of SoS closer to our context, methods based on SoS can simulate AMP (approximate message passing) algorithms, and furthermore makes them robust to certain adversarial corruptions in the input~\cite{sos-amp-robustly} (see also~\cite{sos-gaussian-process,fast-robust-amp}).

Given the many algorithmic successes of SoS, lower bounds against SoS are considered a potential indication of fundamental hardness (subject to the same kinds of caveats we have discussed in Section~\ref{sec:conj}). For the planted clique problem, the landmark result~\cite{pcal} shows failure of SoS in the conjectured hard regime $k \ll \sqrt{n}$, by introducing an approach called \emph{pseudo-calibration}. The proof is extremely technical, but a key bottleneck comes from analyzing low-degree polynomials, specifically the ``advantage'' $\Adv_{\le D}$ from Eq.~\eqref{eq:adv}. The low-degree polynomial framework was born out of this connection, and the first works that used low-degree polynomials to predict thresholds, justified this largely based on the heuristic connection to SoS~\cite{sos-detect,HS-bayesian,hopkins-thesis}. The ``pseudo-calibration conjecture''~\cite[Conjecture~1.2]{sos-detect} (see also~\cite[Conjecture~3.5]{sos-survey}) essentially posits that if low-degree polynomials fail to solve a problem (within a certain well-defined class of high-dim detection problems) then SoS also fails. This is point~(2) above. If the conjecture is resolved affirmatively, it would give a powerful tool for proving SoS lower bounds, since low-degree lower bounds are generally more tractable to prove. Some progress toward the pseudo-calibration conjecture is made in~\cite{sos-detect}, essentially showing that any SoS algorithm can be made into a spectral method on some matrix with low-degree entries (up to some loss in the SNR). Based on the discussion in Section~\ref{sec:alg-captured} it might appear that this spectral method can be implemented using low-degree polynomials (thus proving the conjecture), but actually there is no guarantee of this: for all we know, the spectral method might be poorly conditioned, say, with its maximum positive eigenvalue (which contains the ``signal'') obscured by a huge negative eigenvalue.

These days, low-degree polynomials and SoS exist as two distinct entities. While low-degree lower bounds were first invented in the context of SoS, they can (in hindsight) be motivated more directly without reference to SoS, as we have done in this survey. But how exactly do they compare in power (in the context of high-dim stats)? While I'm not aware of a formal comparison in either direction, I think there is a common tendency to view SoS as a ``stronger'' algorithm and therefore to view SoS lower bounds as a ``stronger'' form of hardness. Point~(3) above is a counterpoint to this, based on the fact that the two classes of algorithms are (in their most basic forms) actually aiming for different objectives, \emph{detection} and \emph{refutation}, which are contrasted in Section~\ref{sec:tasks}. Consider the planted clique result~\cite{pcal} for example. The SoS algorithm is an SDP relaxation that attempts to \emph{refute} the existence of a $k$-clique using the sum-of-squares proof system, and the lower bound says that this refutation fails on input $G(n,1/2)$ when $k \ll \sqrt{n}$. This can naturally be seen as evidence that the \emph{refutation} task is hard, but this need not imply hardness for the (potentially easier) \emph{detection} task. So if we are interested in detection, SoS may not be the best algorithm: it ties its own hands by trying to solve a harder problem than is necessary. This does not actually manifest itself in planted clique, since the detection and refutation thresholds happen to coincide, but there are other settings where SoS will underperform at detection or optimization because it is trying to solve a strictly harder refutation or certification problem; for instance, compare~\cite{affine-planes} with~\cite{opt-sk}, or see~\cite[Section~1.1]{subhypergraph} (namely the comparison to~\cite{sos-densest-k}).

There is also a counterpoint to the above counterpoint: while the main theorem of~\cite{pcal}, taken at face value, should not be considered evidence for hardness of \emph{detection}, the proof implicitly contains a low-degree lower bound for detection. The same is true for other SoS lower bounds that use the pseudo-calibration approach, e.g.,~\cite{KMOW,lifting-sos,affine-planes,sos-machinery,sos-sparse-indep,sos-densest-k,sos-ultra,sos-ngca}: these are implicitly leveraging a hard testing problem and (more-or-less) contain a low-degree lower bound for this problem, even if this is not reflected in the main theorem statement. On a technical level, analyzing $\Adv_{\le D}$ corresponds to controlling the pseudo-expectation value of 1, while proving an SoS lower bound requires the additional step of showing the entire moment matrix is PSD which is much more challenging. In this sense, SoS lower bounds proved via pseudo-calibration are at least as strong as low-degree lower bounds.

The previously mentioned underperformance of SoS for certain detection or optimization tasks is by now well documented, and there have been proposals for how to ``fix'' SoS so that it performs on par with the best low-degree algorithms (and with added ``robustness''). The \emph{local statistics hierarchy} is a modification of SoS suited to detection tasks~\cite{local-stats}, and it has recently been adapted to optimization as well~\cite{sos-amp-robustly}. Lower bounds against local statistics have been shown in a few settings, and these do not appear to have a formal comparison to low-degree lower bounds in either direction. Curiously, local statistics lower bounds seem more tractable than their low-degree counterparts in settings involving random \emph{regular} graphs~\cite{local-stats,spectral-planting,cert-lift-mono} (although these examples only rule out SoS degree 2).

One additional caveat is that each SoS lower bound pertains to some specific choice of how to formulate the problem as a system of constraints. Sometimes there can be more than one natural formulation, and these may potentially be different in power, as discussed in~\cite[Section~1.5]{KM-coloring}.

Overall, provided we are careful about differentiating tasks such as detection versus refutation, the low-degree and SoS frameworks both appear to make the same, ``reliable,'' predictions for computational thresholds. Low-degree lower bounds have the advantage of being easier to prove, making them more abundant and often sharper. SoS lower bounds have the advantage of potentially being perceived as substantially stronger evidence for hardness, depending on your views.

\subsection{Statistical Query Model}

In the \emph{statistical query (SQ) model}~\cite{kearns}, the algorithm is not granted access to the entire input, but is restricted to asking certain types of questions (``queries'') about it. This model can be defined only in scenarios where the input consists of a list of samples (vectors) that are drawn i.i.d.\ from some unknown distribution. The unknown distribution often depends on some planted signal that we aim to recover. SQ algorithms do not have access to individual samples but can query the average value of a function over the samples and receive the answer up to some adversarial error of bounded size. The SQ framework can be used to make predictions about the computational complexity of recovering the signal, based on the number and accuracy of queries required. Some classical works include~\cite{sq-clique,sq-csp,DKS-sq}, and we refer to~\cite[Ch~8]{robust-stat-book} for an overview.

The SQ framework has found many successes in ``learning'' problems, where the goal is to learn an unknown function $f$ (within some class) given a list of $(x,f(x))$ pairs; see e.g.~\cite{sq-gd,sq-massart,sq-list-reg,sq-two-layer,sq-weaker,truncated-gaussian,sq-half} for some recent examples. I feel that the types of problems where SQ is traditionally applied are somewhat different in style from those I have focused on in this survey, although there is definitely some overlap. A problem like planted clique cannot be directly addressed in the SQ framework because the input does not consist of i.i.d.\ samples. However, a bipartite version of the clique problem can be cast in the SQ model, and predicts the ``correct'' threshold $k \approx \sqrt{n}$~\cite{sq-clique}. A different approach to casting planted clique and related problems in the SQ model involves taking each sample to be a ``diluted'' copy of the entire input, with lower signal strength~\cite{sq-tensor-pca,sq-ld}.

For detection problems, the work of~\cite{sq-ld} gives an equivalence result between SQ and low-degree predictions. More accurately, under reasonable assumptions on the problem setting, this result shows equivalence between two \emph{lower bound methods}, namely the \emph{statistical dimension (SDA)} for SQ and the advantage~\eqref{eq:adv} for low-degree. This gives a useful tool for transferring hardness results from one framework to the other. The implication from SQ hardness to low-degree hardness has somewhat milder assumptions than the other direction.

On the other hand, the SQ and low-degree \emph{models} themselves are not really equivalent. It was discovered in~\cite{sq-csp} that the SQ framework makes the ``wrong'' prediction for planted CSPs (constraint satisfaction problems), with SQ algorithms requiring quadratically more constraints than the best known poly-time algorithms. This same work introduced some new, stronger variants of the SQ model (``MVSTAT'' and ``1-MSTAT'') to remedy this. To my knowledge, these have not been widely used since then. The SQ model (in its original ``VSTAT'' form) also struggles with tensor PCA, again with SQ algorithms underperforming~\cite{sq-tensor-pca}. This issue essentially occurs because computing the spectral norm of a matrix requires more samples with the VSTAT oracle than it would with direct access to the samples.

Overall, SQ appears to be a less ``reliable'' predictor than low-degree polynomials for the style of problems we have focused on in Section~\ref{sec:high-dim}, although it does seem reliable for the coarser question of differentiating polynomial versus super-polynomial sample complexity in learning problems. I am not aware of SQ results that capture some of the phenomena we have discussed in Section~\ref{sec:hope}, namely sharp thresholds (such as the BBP transition in the spiked Wishart model) or detection-recovery gaps. Existing SQ lower bounds are stated for either detection or recovery, and the proofs crucially leverage SQ-hardness of testing against the so-called ``reference distribution'' (which plays the role of our ``null'' model). Curiously, there can be problems that are ``easy'' in the SQ model but believed not to admit a poly-time algorithm, e.g.~\cite[Section~1.1]{hard-robust-mean}.

One strength of the SQ model is its ability to capture a variety of algorithmic results. SQ algorithms are allowed arbitrary computation and are only limited by their query access to the input. This is quite different from low-degree polynomials, so it is significant that low-degree lower bounds can be used to rule out these types of algorithms via~\cite{sq-ld}. SQ lower bounds are commonly cited as ruling out various popular approaches including certain convex programs~\cite{sq-cvx} as well as gradient descent~\cite{sq-gd}. The meaning of these claims is that the \emph{standard analysis} of these methods can be carried out within the SQ framework, but SQ lower bounds do not rigorously imply that these methods fail. Together with~\cite{sq-ld}, this is (in some sense) a pathway to linking low-degree polynomials with e.g.~certain convex programs. However, the SQ model does not capture SoS because the feasible region of an SoS program depends on the individual samples, which are not observed.

\subsection{Statistical Physics and Approximate Message Passing}
\label{sec:relations-amp}

As discussed previously in Section~\ref{sec:alg-captured}, \emph{approximate message passing (AMP)} is a framework for building iterative algorithms that have been widely successful in achieving state-of-the-art performance for both planted estimation problems and non-planted optimization problems. A related suite of tools associated with ideas from statistical physics, has proved useful for predicting computational (and statistical) phase transitions in high-dim stats. These tools include the related \emph{belief propagation (BP)} algorithm along with the \emph{cavity} and \emph{replica} methods. A few landmark results in this area include the prediction of the Kesten--Stigum (KS) transition in the stochastic block model~\cite{decelle} and the replica formula in spiked matrix estimation~\cite{LKZ-sparse,MMSE} (with the information-theoretic aspects later proven rigorously~\cite{pf-replica,adaptive-interp,short-pf-replica}), as well as the computational threshold for spin glass optimization~\cite{opt-spin-glass}. See~\cite{info-phys-book,ZK-phys-survey,notes-phys,GMZ-phys-survey,phys-notes} for some general references.

The predictions of computational phase transitions by these methods essentially posit optimality of BP or AMP algorithms. These predictions are often unmatched in their sharpness, pinning down sharp thresholds and exact limiting values for the estimation error (e.g., mean squared error). As discussed in Section~\ref{sec:alg-captured}, we generally expect that polynomials of degree $O(\log n)$, or sometimes $O(1)$, capture BP- or AMP-style algorithms, or the spectral methods obtained from linearizing them (e.g.~\cite{spectral-redemption,bethe-hessian,spectral-phase-retrieval,bp-unstable}). However, this has not been verified in high generality, and still needs to be checked on a problem-by-problem basis, e.g.~\cite{MW-amp,sos-amp-robustly,fourier-iterative,alg-universality}. The broadcast tree model gives one counterexample to this, since BP is not captured by low-degree polynomials in that setting~\cite{KM-tree,ld-bot,ld-bot-opt}.

In a number of settings, AMP remains the best known poly-time algorithm in a sharp sense, so the predictions made by this framework appear reliable. On the other hand, there are a few very classical models where physics-based methods struggle, namely tensor PCA and planted clique. For tensor PCA, the AMP algorithm derived in~\cite{RM-tensor-pca} is shown to be strictly suboptimal compared to certain spectral methods. In planted clique, AMP reaches a sharp threshold $k = \sqrt{n/e}$ in nearly linear time~\cite{amp-clique}, but other (slower, but still poly-time) algorithms reach $k = \epsilon \sqrt{n}$ for an arbitrary constant $\epsilon > 0$~\cite{alon-clique}. In general, AMP seems great for nailing down sharp thresholds but tends to struggle with smooth tradeoffs between SNR (signal-to-noise ratio) and runtime (see Remark~\ref{rem:sharp-smooth}) because it does not include a ``knob'' for increasing the statistical power at the expense of longer runtime. One attempt to add such a ``knob'' is based on the \emph{Kikuchi hierarchy}~\cite{kik51,gen-bp}, also rooted in statistical physics, which ``redeems'' AMP and reaches the optimal SNR-runtime tradeoff in tensor PCA~\cite{kikuchi}. So far, the Kikuchi approach has been limited to tensor PCA and a few related problems (such as CSP refutation~\cite{kikuchi,smoothed-csp,simple-moore}), but not, say, planted clique. The difficulty in extending this approach stems from non-existence of a ``trivial stationary point'' at which to compute the Hessian of the Kikuchi free energy.

In light of the above examples, the low-degree framework is perhaps a more comprehensive and reliable predictor for a wider class of models and scaling regimes. Still, the physics-based methods have a lot to offer, as they give very sharp predictions with relative ease in many situations where our low-degree lower bounds are still lagging behind. It is an ongoing challenge to corroborate some of these physics predictions with low-degree evidence. For instance, in the spiked Wigner model we now know that the exact mean squared error of AMP is optimal among constant-degree polynomials~\cite{MW-amp}, but it remains open to extend this to higher degree. The proof uses the fact that AMP-style algorithms correspond to a certain sub-class of polynomials that are ``tree-structured'' (see also~\cite{univ-polytope,fourier-iterative}).

The physics-based approaches also complement the low-degree framework in that they give a clear (and rather different) explanation for what makes these problems hard. This explanation is ``geometric'' in nature rather than ``algebraic.'' Notably, the performance of AMP can be visualized in terms of a \emph{(``replica-symmetric'') free energy landscape}, where the algorithm's downhill trajectory toward the optimal solution may or may not be impeded by a \emph{free energy barrier}. This is illustrated in Figure~1 of~\cite{pf-replica} or Figure~1 of~\cite{notes-phys}. It would be very interesting to uncover a deeper reason why the geometric and algebraic predictions often coincide. Some partial progress is given by~\cite{fp}, which builds a rigorous connection between the low-degree advantage~\eqref{eq:adv} and a quantity called the \emph{annealed Franz--Parisi potential}, which we will call FP for short. The quantity FP intuitively captures the idea behind the free energy picture, as explained in~\cite[Section~1.3]{fp}, but it is not literally the same quantity that physicists would use and is not directly related to AMP (nor is it quite the same as the annealed potential originally considered by Franz and Parisi~\cite{fp98}). In fact, FP is more closely tied to detection rather than recovery, due to its connection with the low-degree advantage. More recently, a variation of FP has been shown to be directly connected to SQ lower bounds~\cite{fp-sq}, and this applies to a broader class of problems than the original work~\cite{fp} that introduced FP.

Finally, it has been shown that AMP can be used to (approximately) solve certain convex optimization problems on random data, including the popular LASSO method for sparse linear regression~\cite{lasso-risk,lasso-amp,amp-cvx-nonsep,slope}. AMP is captured by low-degree polynomials so we can claim that, at least on some heuristic level, these convex programs are ruled out by low-degree lower bounds. Perhaps this connection can be made more rigorous.

\subsection{Overlap Gap Property}

A related, but formally distinct, approach to the free energy landscapes mentioned above is the \emph{overlap gap property (OGP)}, another geometric property that impedes the progress of certain algorithms. There are various forms of OGP tailored to different settings, and we refer to~\cite{ogp-survey,turing-survey} for a survey. OGP has been used both in non-planted optimization settings and planted recovery settings, and it will be important to differentiate these two cases below.

\subsubsection{Non-Planted OGP}

First we discuss non-planted problems, which was historically the first setting where OGP emerged: in the context of searching for large independent sets in sparse random graphs, it was discovered by~\cite{GS-ogp} that any pair of ``good'' solutions (i.e., independent sets above a certain size) must either have ``small'' or ``large'' overlap, with an ``overlap gap'' in between. Variations of this property have been used to prove failure of algorithms that are sufficiently ``stable'' under perturbations to the input. Notably, certain variants of OGP can be used to rule out low-degree polynomials, e.g.~\cite{ld-opt,indep-set,opt-ksat,barriers-perceptron,strong-ld}, and this is currently the \emph{only} technique we have for proving low-degree lower bounds in non-planted settings. Other algorithms that have been ruled out using OGP include local algorithms on sparse graphs~\cite{GS-ogp,RV-half-opt,ogp-maxcut}, AMP~\cite{ogp-amp,branching-ogp}, low-depth boolean circuits~\cite{ld-opt}, and online algorithms~\cite{barriers-discrepancy,ogp-graph-alignment,online-indep}. Some recent examples~\cite{some-easy-ogp} show that OGP is present even in some ``easy'' problems such as computing shortest paths, meaning that OGP may not be a reliable predictor of hardness on its own. However, in these particular examples, OGP does not rule out degree-$O(\log n)$ polynomials (and in fact, such polynomials succeed), so these are not ``counterexamples'' for the low-degree framework.

\subsubsection{Planted OGP}

Now we move on to planted problems, where the situation is quite different. A natural notion of ``planted'' OGP was defined by~\cite{GZ-ogp}, and the same authors discovered that it surprisingly predicts the ``wrong'' threshold\footnote{Technically, they do not establish OGP below this threshold, but conjecture this based on first-moment behavior. In a different planted model, sparse tensor PCA, OGP has been established in part of the easy regime~\cite{geometric-rule}.} $k \approx n^{2/3}$ for planted clique~\cite{GZ-clique}. This \emph{planted} version of OGP does not rule out low-degree polynomials, but it does rule out certain ``local search'' algorithms when $k \ll n^{2/3}$, showing that these algorithms are strictly suboptimal. Here the meaning of ``local'' is that the algorithm's state is a subset of $k$ vertices, and this is iteratively updated to a ``nearby'' subset in a way that attempts to increase the density of edges among those vertices. (This is distinct from the other meaning of ``local'' we have mentioned for sparse graphs in Section~\ref{sec:alg-captured}.) Furthermore,~\cite{GZ-clique} discovered that ``overparameterization'' fixes the issue and moves the OGP threshold to the ``correct'' computational threshold $k \approx \sqrt{n}$; this effectively means considering local search algorithms that move around on subsets of size $k' \gg k$.

The planted version of OGP, along with the related notion of \emph{free energy barriers}, complements the low-degree framework by ruling out a different class of algorithms than the ones excluded by low-degree lower bounds. For instance, planted OGP can rule out greedy search and gradient descent, along with Markov chains and Langevin dynamics (which are ``noisy'' versions of greedy search and gradient descent, respectively)~\cite{GZ-ogp,alg-tensor,GZ-clique,submatrix-ogp,ogp-sparse,grp-ogp}. I am not aware of any ``natural'' problem where low-degree polynomials are beaten by one of these local search algorithms, but local search algorithms are very popular and so it is valuable to rigorously understand their limits. As in the planted clique example above, a common theme is that the most ``naive'' local search algorithm might be suboptimal, but often there is a simple ``fix'' that gives optimal performance (matching low-degree algorithms). To give another example, gradient descent is suboptimal for tensor PCA~\cite{alg-tensor}, but \emph{averaged gradient descent} provides a fix~\cite{iron-rough}. See also~\cite{local-sparse-tensor} for other ways to fix local search methods. In general, the ``right'' fix might be problem-specific, so it may be difficult to use local search algorithms as a ``reliable'' predictor on their own. For a certain class of problems with additive noise, we know that the naive local search algorithms cannot beat the low-degree threshold~\cite[Section~2.2]{fp}, and furthermore an intriguing ``geometric rule'' appears to predict exactly how suboptimal the naive local search algorithms are~\cite{geometric-rule}.

\subsubsection{More on Markov Chains and Local Search}

Analyzing the performance of Markov chains for average-case inference problems is a difficult technical challenge, both in terms of upper and lower bounds. One difficulty is that the relevant notion of finding a good solution (``hitting time'') is not the same as the commonly studied notion of ``mixing time.'' Until recently, our tools for proving positive results were somewhat lacking, as illustrated by the story of planted clique: the influential work of Jerrum~\cite{jerrum} first predicted the threshold $k \approx \sqrt{n}$ due to failure of a particular Markov chain below it, and it was assumed (but not proven) that this Markov chain should succeed above the threshold; however, recent work shows that in fact this Markov chain fails for all $k \ll n$~\cite{clique-elude}. Some progress in positive results for Markov chains (and Langevin dynamics) has been made by~\cite{alg-tensor,clique-mcmc,locally-stationary,geometric-rule}. On the side of lower bounds, many existing results only rule out Markov chains with worst-case initialization, e.g.~\cite{alg-tensor,GZ-clique,submatrix-ogp,ogp-sparse,grp-ogp} rather than random initialization or a specific natural initialization. One recent work~\cite{clique-elude} overcomes this, showing failure from the natural ``empty'' initialization.

Finally, a different way of reasoning about local search algorithms such as gradient descent is to study the existence (or not) of ``spurious'' local optima in some appropriate non-convex objective function. An early example was~\cite{complexity-spin} and by now there are many more; see~\cite{nonconvex-book} for an overview. Non-existence of such spurious optima generally implies success for a local search algorithm, but curiously, sometimes local search algorithms succeed even in the presence of spurious optima~\cite{passed}.

\subsection{Summary}
\label{sec:summary}

It has certainly been fruitful to study statistical-computational gaps from all of these different complementary viewpoints, and I hope to see this continue. I will also take this opportunity to advocate for why I am especially excited about the low-degree framework, although this comes with the disclaimer that I am likely biased by my own perception of which high-dim stat questions are most ``natural'' or important. Among the unconditional lower bounds, I would say low-degree polynomials are in the top tier for reliable predictions, and also unmatched in their flexibility, covering different types of models, different tasks (detection, recovery, etc.), and different types of phenomena (sharp thresholds, precise super-polynomial runtimes, etc.). When a new problem comes along, the low-degree framework tends to be a user-friendly option to try first, although there are some settings where AMP-based predictions are easier and sharper, and I would like to see low-degree polynomials ``catch up'' here!

\begin{table}
\centering
\begin{tabular}{|l|x{1cm}|x{1cm}|x{1cm}|x{1cm}|x{1cm}|} 
\hline
 & SoS & SQ & AMP & OGP & LDP \\ 
\hline
Detection & \checkmark & \checkmark &  &  & \checkmark \\
\hline
Recovery &  & {\tiny\checkmark} & \checkmark & {\tiny\checkmark} & \checkmark \\
\hline
Optimization &  &  & \checkmark & \checkmark & \checkmark \\
\hline
Refutation & \checkmark &  &  &  & \checkmark \\
\hline
\end{tabular}
\captionsetup{width=.95\linewidth}
\caption{Which of the frameworks covered in Section~\ref{sec:relations} can (thus far) give unconditional hardness results for which tasks? The frameworks are --- with representative citations for the entries in the table --- sum-of-squares~(SoS)~\cite{local-stats,pcal}, statistical queries~(SQ)~\cite{sq-clique}, approximate message passing~(AMP)~\cite{LKZ-sparse,opt-spin-glass}, overlap gap property~(OGP)~\cite{GZ-ogp,GS-ogp}, and low-degree polynomials~(LDP)~\cite{HS-bayesian,SW-estimation,ld-opt,coloring-clique}. The two small check marks signify that (1) known SQ lower bounds for recovery are always leveraging a hard testing problem, so cannot separate recovery from detection, and (2) OGP for planted problems is useful for probing the limits of local search methods but is not a reliable predictor of the true computational threshold. Reductions are a separate endeavor, but these have been known to cover all four tasks if using the appropriate starting assumptions~\cite{BBH-reduction,VV-perceptron,det-rec-reduction}.}
\label{table}
\end{table}

SoS is also a strong contender for reliable predictions, subject to the caveats discussed in Section~\ref{sec:sos}. In fact, some might regard SoS lower bounds as \emph{significantly} stronger evidence for hardness than low-degree lower bounds. Personally I am not convinced of this --- but this is quite subjective and my opinion might be controversial here! --- and I value the low-degree framework's ability to formulate different tasks (detection, recovery, etc.)\ in a way that is particularly transparent and interpretable, compared to SoS.

We have also seen that the different frameworks are not entirely disjoint: for instance, SoS lower bounds tend to contain low-degree lower bounds, and low-degree lower bounds for optimization are always proved by some form of OGP. Exploring the connections between frameworks is an important ongoing quest, which has already seen some successes. We recap some of the highlights here, with the warning that these statements are oversimplifications of the actual results.

\begin{itemize}
    \item SoS algorithms can be turned into spectral methods~\cite{sos-detect}.
    \item Spectral methods can be turned into low-degree polynomials (Remark~\ref{rem:log-deg}).
    \item Physics can inform the design of spectral methods, e.g.~\cite{spectral-redemption,bethe-hessian,kikuchi,spectral-phase-retrieval}.
    \item OGP rules out AMP~\cite{ogp-amp} and low-degree polynomials~\cite{ld-opt} in non-planted settings.
    \item Certain lower bound approaches in the SQ and low-degree models are equivalent~\cite{sq-ld}.
    \item The Franz--Parisi potential connects geometric and algebraic hardness~\cite{fp,fp-sq}.
    \item AMP and low-degree polynomials are equivalent in some settings~\cite{MW-amp}.
    \item SoS simulates AMP, robustly~\cite{sos-amp-robustly}.
\end{itemize}

Most of these results have non-trivial caveats and technical conditions that we have swept under the rug, for instance the first two statements cannot be ``chained'' together as explained in Section~\ref{sec:sos}, and many of these results require assumptions on the problem setting. Perhaps future work will manage to strengthen these connections and build a more unified theory. On the other hand, we cannot hope to simply prove equivalence between all the frameworks, because the reality is more complicated: we have seen various examples where they simply do not make the same predictions, e.g., tensor PCA ``fools'' SQ, AMP, and OGP~\cite{sq-tensor-pca,RM-tensor-pca,geometric-rule}. In fact, often the different frameworks are not even asking the same question, due to discrepancies between the different tasks discussed in Section~\ref{sec:tasks}. Table~\ref{table} summarizes which frameworks can give lower bounds for which tasks. The low-degree framework is unique in that it can be applied to all four tasks, putting it in a central position for tying together the other frameworks. There is still more to understand about these connections. For instance, can low-degree lower bounds be shown to exclude algorithms like convex programs or Markov chains? Finally, it is exciting to see average-case reductions expand their sphere of influence, as this puts our predictions on more rigorous footing, without relying on a ``low-degree conjecture.''

\section{Proof Techniques}
\label{sec:techniques}

This is intended as a non-technical overview of the existing techniques for proving low-degree lower bounds, with references to the literature for more details.

\subsection{Detection}
\label{sec:pf-det}

We first consider detection problems where the goal is to test between a planted distribution $\PP = \PP_n$ and a null distribution $\QQ = \QQ_n$, where, crucially, $\QQ$ has independent coordinates. This is the most basic setting for low-degree lower bounds, which appeared in the earliest works on the low-degree framework~\cite{sos-detect,HS-bayesian,hopkins-thesis}. We will discuss more complicated choices for $\QQ$ in Section~\ref{sec:planted-v-planted-pf}. Throughout, assume that for each $n$, $\PP_n$ and $\QQ_n$ have finite moments of all orders so that polynomials have finite expectation; these moments may have arbitrary dependence on $n$.

\subsubsection{Reformulation as Advantage}
\label{sec:adv}

To prove a lower bound, the goal is to rule out strong or weak separation, as defined in Definition~\ref{def:sep}. To do this, it suffices to bound a certain ratio, which we have already introduced under the name ``advantage'' in~\eqref{eq:adv}.

\begin{proposition}\label{prop:adv-to-sep}
Consider planted and null distributions $\PP = \PP_n$ and $\QQ = \QQ_n$. For some $D = D_n$, let
\[ \Adv_{\le D} := \sup_{f \in \RR[Y]_{\le D}} \frac{\EE_{Y \sim \PP}[f(Y)]}{\sqrt{\EE_{Y \sim \QQ}[f(Y)^2]}}. \]
\begin{itemize}
\item If $\Adv_{\le D} = O(1)$ as $n \to \infty$ then no degree-$D$ polynomial strongly separates $\PP$ and $\QQ$.
\item If $\Adv_{\le D} = 1+o(1)$ as $n \to \infty$ then no degree-$D$ polynomial weakly separates $\PP$ and $\QQ$.
\end{itemize}
\end{proposition}

\noindent Note that $\Adv_{\le D} \ge 1$ always, due to the choice $f(Y) \equiv 1$. The proof of Proposition~\ref{prop:adv-to-sep} is a basic reformulation of the definition of separation. The details can be found, for instance, in the proof of Lemma~4.1 in~\cite{coloring-clique}. The upshot is that we now aim to compute or bound the quantity $\Adv_{\le D}$ but we first pause to discuss the history of this quantity.

\begin{remark}[Advantage]
\label{rem:adv}
The quantity we have denoted $\Adv_{\le D}$ (for ``advantage'') has more traditionally been denoted $\|L^{\le D}\|$ (or variants such as $\|LR^{\le D}\|$) due to its characterization as the \emph{norm of the low-degree likelihood ratio}. We omit the precise meaning of this and refer the reader to Sections~1.2--1.3 of~\cite{ld-notes} for this viewpoint and its relation to classical statistics, where the \emph{likelihood ratio} $L := d\PP/d\QQ$ and its norm $\|L\| := \sqrt{\EE_{Y \sim \QQ}[L(Y)^2]}$ have similar implications for information-theoretic indistinguishability. The notation $\Adv_{\le D}$ is a convenient alternative for situations where one does not wish to define likelihood ratios and norms, etc. In fact, $\Adv_{\le D}$ may be well-defined even in situations where the likelihood ratio does not exist, as in~\cite{planted-vec}. The term ``low-degree likelihood ratio'' was coined by Hopkins in~\cite[Ch~2]{hopkins-thesis}, where the connections to classical statistics were also first articulated. That work takes a different convention from us by centering polynomials under $\QQ$. Their quantity $\|LR^{\le D}-1\|$ is, in our notation,
\[ \sup_{f \in \RR[Y]_{\le D}} \frac{\EE_{Y \sim \PP}[f(Y)] - \EE_{Y \sim \QQ}[f(Y)]}{\sqrt{\EE_{Y \sim \QQ}[f(Y)^2]}} = \sup_{\substack{f \in \RR[Y]_{\le D} \\ \EE_{Y \sim \QQ}[f(Y)] = 0}} \frac{\EE_{Y \sim \PP}[f(Y)]}{\sqrt{\EE_{Y \sim \QQ}[f(Y)^2]}} = \sqrt{\Adv_{\le D}^2 - 1}. \]
To see the second equality, compare Eq.~(2.3.1) of~\cite{hopkins-thesis} with our Eq.~\eqref{eq:adv-c} below. The quantity $\Adv_{\le D}^2 - 1$ is analogous to the \emph{chi-squared divergence} (see e.g.~\cite{sparse-adversarial,sparse-tensor,grp-testing}) and can be denoted $\chi^2_{\le D}(\PP \| \QQ)$. The term ``advantage'' sometimes carries a different meaning in the context of hypothesis testing, and the connection to our usage is explained in~\cite{optimal-distinguishers-pc}.
\end{remark}

\subsubsection{Computing the Advantage}

Returning now to the task of computing $\Adv_{\le D}$, a key insight that arose from~\cite{pcal,sos-detect,HS-bayesian} is that $\Adv_{\le D}$ can be calculated explicitly using orthogonal polynomials.

\begin{proposition}\label{prop:adv-c}
Suppose $h_0,\ldots,h_m \in \RR[Y]_{\le D}$ form a basis for $\RR[Y]_{\le D}$ (as a vector space over $\RR$) and furthermore suppose they are orthonormal with respect to $\QQ$, meaning
\[ \Eop_{Y \sim \QQ} [h_i(Y) h_j(Y)] = \One_{i = j}. \]
Then
\begin{equation}\label{eq:adv-c}
\Adv_{\le D}^2 = \sum_{i=0}^m \left(\EE_{Y \sim \PP}[h_i(Y)]\right)^2.
\end{equation}
\end{proposition}

\begin{proof}
There are a few different ways to prove this. The more traditional proof involves projecting the likelihood ratio onto the subspace of degree-$D$ polynomials, and we refer to~\cite[Section~2.3]{hopkins-thesis} or~\cite[Section~2.3]{ld-notes} for this viewpoint. We give here a self-contained proof that does not require the likelihood ratio to exist. Expand a candidate polynomial $f$ as $f(Y) = \sum_{i=1}^m \hat{f}_i \, h_i(Y)$ for some vector of real coefficients $\hat{f} = (\hat{f}_0,\ldots,\hat{f}_m)$. Due to orthogonality of $\{h_i\}$,
\[ \Eop_{Y \sim \QQ}[f(Y)^2] = \sum_{i=0}^m \hat{f}_i^2 = \|\hat{f}\|^2. \]
Defining the vector $c = (c_0,\ldots,c_m)$ with $c_i = \EE_{Y \sim \PP}[h_i(Y)]$, we have
\[ \Eop_{Y \sim \PP}[f(Y)] = \sum_{i=0}^m \hat{f}_i \Eop_{Y \sim \PP}[h_i(Y)] = \langle c,\hat f \rangle. \]
Putting it together,
\[ \Adv_{\le D} = \sup_{\hat f} \frac{\langle c,\hat f \rangle}{\|\hat{f}\|} = \|c\|, \]
since the optimizer is $\hat{f} = c$ (or any scalar multiple thereof).
\end{proof}

\subsubsection{Orthogonal Polynomials}

From above, if we can construct orthogonal polynomials for $\QQ$, then we can compute $\Adv_{\le D}$ using the explicit formula~\eqref{eq:adv-c}. When $\QQ$ has independent coordinates, it is generally tractable to find an explicit basis of orthogonal polynomials (which is what makes this case especially approachable). We mention some useful examples of orthogonal polynomials below.

\begin{itemize}
\item If $\QQ$ has entries $Y_1,\ldots,Y_N$ that are i.i.d.\ and uniform on $\{\pm 1\}$, the orthogonal polynomials are the usual boolean Fourier characters (see~\cite{boolean-book}), namely, for $S \subseteq [N]$,
\[ \chi_S(Y) = \prod_{i \in S} Y_i. \]
For a given degree $D$, the polynomials $\{\chi_S\}_{|S| \le D}$ form an orthonormal basis for $\RR[Y]_{\le D}$. This covers the case where the null distribution is the random graph $G(n,1/2)$, where $N = \binom{n}{2}$ and $+1,-1$ encode the presence or absence of an edge, respectively.

\item If $\QQ$ has entries $Y_1,\ldots,Y_N$ that are i.i.d.\ Bernoulli$(p)$, the orthogonal polynomials are, for $S \subseteq [N]$,
\[ \chi_S(Y) = \prod_{i \in S} \frac{Y_i-p}{\sqrt{p(1-p)}}. \]
Again, $\{\chi_S\}_{|S| \le D}$ forms an orthonormal basis for $\RR[Y]_{\le D}$~\cite{janson-orthog-book}.

\item If $\QQ$ has entries that are i.i.d.\ $\cN(0,1)$, the orthogonal polynomials are the \emph{multivariate Hermite polynomials}. These are commonly used in low-degree lower bounds; see, for instance,~\cite[Section~2.3]{ld-notes}. Note that we need to normalize these so that they are orthonormal with respect to $\cN(0,1)$, which is not always the standard convention for Hermite polynomials.

\item The above examples are all special cases of the more general \emph{Morris class of exponential families}, discussed in~\cite{morris}.

\item More generally, if $\QQ$ has independent coordinates, one can find (univariate) orthogonal polynomials for each coordinate $i$ by applying Gram--Schmidt to the monomial basis $\{1,Y_i,Y_i^2,\ldots\}$, and these can be extended to a multivariate basis by taking products of the univariate basis elements.

\item A more involved example is the case where $\QQ$ is uniform over $k$-sparse binary vectors in $\{0,1\}^N$, in which case orthogonal polynomials are given by~\cite{filmus}. These are used for a low-degree lower bound in~\cite[Section~8.2]{grp-testing}.

\item Instead of a truly orthogonal basis, it suffices to use a near-orthogonal basis for $\QQ$. One example is the ``tensor cumulants'' of~\cite{tensor-cumulants}, which form a near-orthogonal basis for rotationally-invariant functions of tensors.

\end{itemize}

\subsubsection{Putting it Together}
\label{sec:together}

Provided we have explicit orthogonal polynomials $\{h_i\}$ for $\QQ$, the rest of the proof amounts to evaluating the expectations $\EE_{Y \sim \PP}[h_i(Y)]$ and then bounding the formula~\eqref{eq:adv-c}. We mention a few different strategies for going about this.

\begin{itemize}

\item The most basic approach is to find an explicit formula for $\EE_{Y \sim \PP}[h_i(Y)]$ and then bound the formula~\eqref{eq:adv-c} by direct combinatorial arguments. An instructive example of this is the case of planted clique, covered in~\cite[Section~2.4]{hopkins-thesis}. The computation of $\EE_{Y \sim \PP}[h_i(Y)]$ is easy for basic models like planted clique, but can become much more difficult for more complicated models such as random geometric graphs~\cite{kiril-fourier}. The second step --- bounding~\eqref{eq:adv-c} --- can also lead to some involved combinatorics. In cases where the direct approach seems unwieldy, the methods below offer some potential shortcuts.

\item In some cases, $\Adv_{\le D}$ admits an elegant formula. For instance, consider the additive Gaussian model where we aim to distinguish $\PP: Y = X + Z$ from $\QQ: Y = Z$ where $Z$ is an i.i.d.\ $\cN(0,1)$ vector (or matrix, etc.)\ and $X$ is a random vector from some arbitrary distribution (and independent from $Z$). Here we have~\cite[Thm~2.6]{ld-notes}
\begin{equation}\label{eq:gaussian-adv}
\Adv_{\le D}^2 = \sum_{d=0}^D \frac{1}{d!} \Eop_{X,X'} \langle X,X' \rangle^d
\end{equation}
where $X'$ is an independent copy of $X$. This formula effectively reduces the question to understanding moments of the ``overlap'' random variable $\langle X,X' \rangle$. See e.g.~\cite[Section~3.1.1]{ld-notes} or~\cite[Section~3.2]{sparse-clustering} for an example of a low-degree lower bound carried out this way. Other settings that admit convenient formulas (or bounds) for $\Adv_{\le D}$ include the spiked Wishart model (both rank-1~\cite[Lemma~5.9]{ld-notes} and higher-rank~\cite[Appendix~A]{spectral-planting} variants), non-Gaussian component analysis~\cite[Lemma~6.4]{planted-vec}, and models with independent noise from the Morris class of exponential families~\cite{morris}.

\item Some models can be related to a corresponding Gaussian model, allowing the formula~\eqref{eq:gaussian-adv} to be used. For instance, certain binary observation models with independent noise can be reduced to the Gaussian case~\cite[Appendix~B.1]{spectral-planting}. More general relations between different noise models are captured by the ``channel monotonicity'' of~\cite{morris} and ``channel universality'' of~\cite{kunisky-coordinate}.

\item Another approach involves a low-overlap truncation, also called the \emph{Franz--Parisi criterion} due to its conceptual link with some ideas based in statistical physics (see~\cite[Section~1.3]{fp}). The formula~\eqref{eq:gaussian-adv} can be viewed as a Taylor series truncation of
\[ \Eop_{X,X'} \exp(\langle X,X' \rangle) \]
which is the formula for $\Adv_{\le \infty}^2$, a.k.a., the squared norm of the likelihood ratio. A different type of truncation arises from restricting the overlap, namely
\begin{equation}\label{eq:fp}
\Eop_{X,X'} \One_{|\langle X,X' \rangle| \le \delta} \, \exp(\langle X,X' \rangle)
\end{equation}
for some choice of $\delta = \delta_n$. For a particular correspondence between $D$ and $\delta$, the two forms of truncation~\eqref{eq:gaussian-adv} and~\eqref{eq:fp} are closely related, so one can bound~\eqref{eq:fp} as a means to control $\Adv_{\le D}$~\cite{fp}. A version of this strategy can be extended beyond the additive Gaussian model, to certain ``planted sparse models''~\cite[Section~3]{fp}. An advantage of this strategy is that~\eqref{eq:fp} may be relatively easy to control, and one can largely avoid working with the details of the orthogonal polynomials, as in~\cite[Section~7.4]{grp-testing}. We refer to~\cite{fp} for further details of this approach, versions of which have also appeared in~\cite{sk-cert,ld-notes,spectral-planting,grp-testing}.

\item A rather different approach taken by~\cite{kunisky-coordinate} uses the \emph{Efron--Stein decomposition}. Instead of ruling out polynomials of degree (at most) $D$, this approach rules out a larger class of functions: those that can be expressed as a linear combination of functions that each depend on only $D$ coordinates (this is the notion of \emph{coordinate degree} from Remark~\ref{rem:coordinate-deg}). The subspace of such functions is the sum of many overlapping subspaces --- one for each choice of $D$ coordinates --- and the principle of inclusion--exclusion can be used to project the likelihood ratio onto this subspace. This method does not rely too heavily on the specific details of the distribution $\QQ$ or its orthogonal polynomials, and is shown in~\cite{kunisky-coordinate} to produce some general results that hold for a variety of different noise channels.

\item It can often be argued that the optimal polynomial must be invariant under the inherent symmetries of the problem. For instance, random graph models tend to be symmetric with respect to permutation of the vertices. In some cases it may be helpful to reduce to symmetric polynomials, as in~\cite{MW-amp,semerjian,tensor-cumulants,some-easy-ogp}.

\end{itemize}

\subsubsection{Conditioning}
\label{sec:cond}

While $\Adv_{\le D} = O(1)$ rules out strong separation, the converse is not always true. That is, there are situations where detection is hard --- in the sense that low-degree polynomials fail to achieve strong separation --- but $\Adv_{\le D} = \omega(1)$, due to some rare ``bad'' event under $\PP$. In these situations, we cannot quite use the above proof strategy, since it relies on showing $\Adv_{\le D} = O(1)$. However, we may still be able to use a \emph{conditional} version of the same strategy: define a modified planted distribution $\tilde{\PP} = \PP | A$ for some ``good'' event\footnote{The event $A$ is allowed to depend on both the sample $Y \sim \PP$ and any latent randomness used to generate it, such as properties of the hidden planted signal.} $A$ that has probability $1-o(1)$ under $\PP$. If we can show $\Adv_{\le D} = O(1)$ for the modified testing problem ($\tilde\PP$ versus $\QQ$), this rules out strong separation in the \emph{original} testing problem ($\PP$ versus $\QQ$); see\footnote{These citations assume $\tilde{P}$ is absolutely continuous with respect to $\QQ$ so that the likelihood ratio is defined. As mentioned in Remark~\ref{rem:adv}, this is not actually needed.}~\cite[Prop~6.2]{fp} or~\cite[Lemma~7.3]{grp-testing}, which also include a similar result for weak separation. While conditioning needs to be used as a proof device in the lower bound, our notion of ``success'' (strong/weak separation for the \emph{original} $\PP$) remains a useful one for both upper and lower bounds, as in~\cite[Thm~3.3]{grp-testing}. Low-degree lower bounds based on conditioning have appeared in sparse linear regression~\cite{fp}, Bernoulli group testing~\cite{grp-testing}, planted dense subgraph~\cite{subhypergraph}, graph matching~\cite{correlated-er}, and planted clique with non-adaptive edge queries~\cite{sublinear-clique}. A more subtle example is the case of correlated block models~\cite{correlated-sbm}, where the authors condition on a constant-probability event rather than a high-probability one.

\subsection{Recovery}

When approached with a recovery question, the easiest form of lower bound to try first is to look for a detection problem (where $\QQ$ has independent coordinates) that ``explains'' the hardness of recovery via a two-stage argument (see Section~\ref{sec:variations-ld}, or the more sophisticated detection-to-recovery reductions in~\cite{alg-contig,conj-sharp-sbm}). This way, the tools from the previous section can be applied. However, in some cases this may not be possible due to a detection-recovery gap (see e.g.~\cite[Appendix~B]{SW-estimation}). Also, we might prefer to find an unconditional bound on $\MMSE_{\le D}$ rather than a two-stage argument. Directly addressing the recovery question is more difficult, but we do have a growing toolbox of methods. While our tools for analyzing detection have a close analogy with classical statistics (e.g., chi-squared divergence), our tools for recovery appear rather different from their statistical counterparts.\footnote{For instance, we do not have a low-degree version of Fano's inequality (see~\cite{fano}), and I suspect this may not be possible because Fano reduces the problem to 2-point testing problems that do not have stat-comp gaps.}

Our goal is to give a lower bound on $\MMSE_{\le D}$, as defined in Definition~\ref{def:mmse}. This can be reformulated as a ratio similar to $\Adv_{\le D}$, namely~\cite[Fact~1.1]{SW-estimation}
\[ \MMSE_{\le D} = \EE[x^2] - \Corr_{\le D}^2 \]
where (recall) $x$ is the scalar quantity we aim to estimate and the \emph{degree-$D$ correlation} is defined as
\[ \Corr_{\le D} := \sup_{f \in \RR[Y]_{\le D}} \frac{\EE[f(Y) \cdot x]}{\sqrt{\EE[f(Y)^2]}}. \]
Here, all expectations are over the joint distribution of $(x,Y)$ according to the \emph{planted} model; there is no null model for a recovery problem.

Akin to Proposition~\ref{prop:adv-c} for $\Adv_{\le D}$, there is an explicit linear-algebraic formula for $\Corr_{\le D}$. Namely, if we choose any basis $h_0,\ldots,h_m$ for $\RR[Y]_{\le D}$ in which to expand a candidate polynomial $f$, and define the vector $c_i = \EE[h_i(Y) \cdot x]$ and matrix $M_{ij} = \EE[h_i(Y) \cdot h_j(Y)]$, then
\[ \Corr_{\le D} = \sup_{\hat f} \frac{\langle c,\hat f \rangle}{\sqrt{\hat f^\top M \hat f}} = \sqrt{c^\top M^{-1} c} \]
where the optimizer is $\hat f = M^{-1} c$. For the types of planted measures we tend to consider, it is possible to compute the entries of $M$ explicitly but it does not seem tractable to explicitly describe the entries of $M^{-1}$. It also does not seem tractable to find an explicit orthogonal basis for $\RR[Y]_{\le D}$ with respect to the \emph{planted} measure (which would make $M$ diagonal, and thus easy to invert). The existing approaches do not attempt to calculate $c^\top M^{-1} c$ directly, but use a variety of tricks to find a tractable upper bound. We now describe some of these approaches.

\begin{itemize}

\item For certain models of the form ``signal plus noise'' (informally speaking), it is possible to bound $\Corr_{\le D}$ in terms of certain joint cumulants of the signal~\cite{SW-estimation}. To derive these formulas, a lower bound on $\EE[f(Y)^2]$ is obtained by applying Jensen's inequality to the ``signal'' part only (not the noise). See also~\cite[Proposition~4.2]{dense-cycles} for a refined version of this formula for binary observation models. Other uses of this method include~\cite{graphon,hd-clustering}.

\item A more powerful framework that generalizes the Jensen trick above involves bounding $\EE[f(Y)^2]$ using orthogonal polynomials \emph{not in $Y$ but in some latent independent random variables used to generate the model}. For instance, in the planted clique model (Definition~\ref{def:pc}), it is not clear how to construct orthogonal polynomials in the observed variables $Y$, but we can construct orthogonal polynomials that are functions of the following independent random variables: a Bernoulli($1/2$) variable for each edge, representing whether or not it is included in $G(n,1/2)$, and a Bernoulli($k/n$) variable for each vertex, representing whether or not it is included in the planted clique. Some version of this strategy appears in~\cite{ld-tensor-decomp}. A more streamlined ``recipe'' is articulated in Section~1.2 of~\cite{sharp-est} (and I also thank Jonathan Niles--Weed for some discussions that helped inspire this). At a high level, the strategy is to construct a random variable $u$ with $\EE[u^2]$ as small as possible, subject to the constraint $\EE[f(Y) \cdot (u-x)] = 0$ for all $f \in \RR[Y]_{\le D}$. This acts as a ``dual certificate'' for $\Corr_{\le D}$ because
\[ \EE[f(Y) \cdot x] = \EE[f(Y) \cdot u] \le \sqrt{\EE[f(Y)^2]} \cdot \sqrt{\EE[u^2]}, \]
implying $\Corr_{\le D} \le \sqrt{\EE[u^2]}$. The random variable $u$ is represented using its expansion in the orthogonal polynomial basis mentioned above, in order to control $\EE[u^2]$. This method is more complicated than the cumulant formula mentioned above, but gives some of the sharpest known results~\cite{sharp-est,sbm-many}.

\item A different approach taken by~\cite{ld-bot,ld-bot-opt} for the broadcast tree model does not use orthogonal polynomials at all. Instead, the term $\EE[f(Y)^2]$ is bounded using an intricate method of decomposing the polynomial $f$.

\item The approach of~\cite{MW-amp} is unique in that it gives \emph{asymptotically exact} bounds on $\MMSE_{\le D}$ (albeit only for constant $D$) in a setting (spiked Wigner model) where $\MMSE_{\le D}$ converges to some non-trivial constant (described by the replica formula). The proof shows that the restricted class of ``tree-structured polynomials'' are as powerful as all polynomials of the same degree, and that tree-structured polynomials can be understood via a connection to existing theory on AMP (the optimal algorithm). See also~\cite{semerjian} for a related approach.

\end{itemize}

\subsection{Other Tasks}

We have focused on detection and recovery so far, but now we discuss low-degree lower bounds for other types of tasks.

\subsubsection{Planted-vs-Planted Testing}
\label{sec:planted-v-planted-pf}

Suppose we aim to hypothesis test between two distributions $\PP$ and $\QQ$ that both contain planted structure. For instance, maybe $\PP$ is a distribution with two planted cliques while $\QQ$ has only one planted clique. While similar to detection, the term ``detection'' is perhaps not appropriate here because our objective is a bit different from ``detecting'' the presence of a planted signal. On a technical level, we can use the same definition of ``success'' as in detection (strong/weak separation), but the analysis is more difficult because $\QQ$ does not have independent coordinates. To prove a lower bound, we aim to bound $\Adv_{\le D}$ as in detection, but now the distribution $\QQ$ appearing in the denominator is a planted distribution. This type of $\Adv_{\le D}$ bears close resemblance to the quantity $\Corr_{\le D}$ appearing in recovery lower bounds, and so the ideas used in recovery lower bounds can also be used for planted-vs-planted testing. We refer to~\cite{planted-v-planted} for further details, and to~\cite{coloring-clique,kiril-fourier} for additional examples.

\subsubsection{Refutation}

As described in Section~\ref{sec:variations-ld}, hardness of refutation can be argued via connection to some associated detection problem. There is also a direct way to formulate refutation as a task for polynomials, and define a notion of separation that captures success~\cite[Section~2.2]{coloring-clique}. Both options lead to similar proof strategies: one needs to construct a planted distribution $\PP$ that contains the structure we aim to refute, and then bound $\Adv_{\le D}$ between $\PP$ and our given distribution $\QQ$~\cite[Prop~2.14]{coloring-clique}. Provided $\QQ$ has independent coordinates, this is done using the usual detection techniques described in Section~\ref{sec:pf-det}. This proof strategy is known to be ``complete'' in the following sense: if refutation is hard for low-degree polynomials then there must exist a hard planted distribution $\PP$ such that $\Adv_{\le D}$ is bounded~\cite[Thm~2.20]{coloring-clique}. We refer to~\cite{coloring-clique} for further details.

We remark on a difference between refutation and recovery here. Hardness of both refutation and recovery can be deduced from hardness of detection via a two-stage argument (see Section~\ref{sec:variations-ld}). As mentioned above, refutation has the convenient property that bounds on $\Adv_{\le D}$ for the detection problem can be directly translated to (unconditional) failure of low-degree refutation algorithms. For recovery, we do not know such an implication: it is not clear how to translate bounds on $\Adv_{\le D}$ for the detection problem to bounds on $\MMSE_{\le D}$ for the recovery problem.

\subsubsection{Optimization}

Non-planted optimization problems pose a rather different set of challenges. Since there are potentially many different solutions achieving a target objective value, it is not clear how to frame the question as a linear algebra problem as we have seen for detection and recovery. Instead, the existing techniques combine some version of the \emph{overlap gap property (OGP)}, a structural property of the solution space, with the fact that low-degree polynomials exhibit some form of ``stability'' to small changes in the input. We refer to~\cite[Thm~1.4]{strong-ld} for some examples of these arguments, and to~\cite{ogp-survey,turing-survey} for a survey on the OGP and its other uses. Some of the more sophisticated uses of OGP involve the so-called \emph{branching OGP}~\cite{branching-ogp}. Originally, the class of algorithms ruled out by branching OGP did not include low-degree polynomials, but this has now been remedied, at least in one setting~\cite{sellke-spherical}.

\subsection{Upper Bounds}

In Section~\ref{sec:comm-upper}, I have attempted to stress the importance of low-degree \emph{upper bounds} that use the same measure of success (separation or mean squared error) as the lower bounds. Section~\ref{sec:list} includes a list of works that prove such results. The proof essentially amounts to constructing a polynomial and bounding its first two moments. However, this may be tedious to do directly, and the following considerations may be helpful.

\begin{itemize}

\item If we know a poly-time algorithm already, it may be possible to approximate this by a polynomial; see Section~\ref{sec:alg-captured} for some examples. On the other hand, algorithms that are sequential or iterative in nature tend to incur a severe blowup in degree (particularly if more than a constant number of iterations are used). In some cases it may be a better idea to design a new polynomial from scratch rather than attempting to mimic some existing algorithm.

\item There are certain tricks to engineering a polynomial in such a way that the variance is controlled or the analysis is easier. This includes, for instance, counting self-avoiding walks instead of all walks~\cite{HS-bayesian,heavy-tailed-saw,dense-cycles,sharp-est}, or using subgraph counts that are ``balanced''~\cite{graph-match-nearly,subhypergraph}. A sophisticated example is the idea of counting chandeliers~\cite{graph-match-otter}.

\item Rather than directly computing the moments of some polynomial, it may be easier to show that it ``succeeds with very high probability'' (appropriately defined), and then deduce conclusions about the first two moments via \emph{hypercontractivity} as in~\cite[Section~4]{SW-estimation}.

\item It may be useful to have a polynomial that approximates a threshold function in some appropriate sense, for instance~\cite[Prop~4.1]{SW-estimation}.

\end{itemize}

\section{Open Problems}
\label{sec:open}

The following open problems and broad directions are, in my view, some of the prominent challenges that we currently face in this area.

\begin{enumerate}

\item {\bf Ruling out concrete algorithms}: One general direction of interest is to expand the reach of the low-degree framework by showing that low-degree lower bounds imply failure of various concrete algorithm classes. One existing example is the connection between low-degree and statistical query algorithms~\cite{sq-ld}. Some algorithms that are \emph{not} known to be ruled out by low-degree lower bounds are Markov chains\footnote{There is some progress for Markov chains~\cite[Section~2.2]{fp}, but this is limited to Gaussian additive models.} and convex programs. The unproven ``pseudo-calibration conjecture''~\cite[Conjecture~1.2]{sos-detect} offers one precise sense in which low-degree lower bounds might rule out sum-of-squares programs.

\item {\bf Other notions of successful detection}: As discussed in Section~\ref{sec:comm-sep}, our existing low-degree lower bounds for detection rule out ``separation'' in terms of the first two moments, but this does not formally preclude a test that succeeds with high probability by thresholding the value of some low-degree polynomial. It may be desirable to rule out other notions of success, such as thresholding. Some weak results of this flavor include~\cite[Thm~4.3]{ld-notes} and~\cite[Section~2.3]{sparse-clustering}, but these include some undesirable technical conditions. A more refined version of this question is to characterize the tradeoff between type I and II errors achievable by low-degree polynomials, as discussed in~\cite[Section~2.4]{precise-error}.

\item {\bf Recovering physics predictions}: As discussed in Section~\ref{sec:relations-amp}, conjectures based on AMP can give exact predictions for the best possible asymptotic MSE achievable by efficient algorithms, in regimes where this MSE converges to some non-trivial constant described by a replica formula. See for instance~\cite{LKZ-sparse,MMSE}, or~\cite[Fig~1]{pf-replica} for an illustration. We are currently lacking low-degree lower bounds that achieve this same level of precision. Some progress has been made for constant-degree polynomials~\cite{MW-amp} but it remains open to extend this to higher degree.

\item {\bf Refutation lower bounds}: Our techniques for proving low-degree lower bounds for refutation are somewhat lacking. Refuting $k$-colorability in $G(n,1/2)$ is one very basic example of a problem where our low-degree upper and lower bounds do not match~\cite{coloring-clique}. Existing techniques rely on constructing a ``quiet'' planted distribution, which tends to be done in an ad hoc way.

\item {\bf Settings without independent randomness}: Most of our existing tools for low-degree lower bounds rely on exploiting independent random variables and their orthogonal polynomials, although one exception is~\cite{ld-bot,ld-bot-opt}. Problems involving random regular graphs, e.g.~\cite{cert-lift-mono}, are one setting where we are currently lacking the tools to prove low-degree lower bounds.

\item {\bf Low-degree upper bounds}: In various settings, we have low-degree lower bounds that match the best known algorithms, but we are missing a matching \emph{low-degree upper bound} (and as usual, I am slightly picky about what counts here, as explained in Sections~\ref{sec:comm-sep}--\ref{sec:comm-upper}). Establishing these would help to give credibility to the low-degree framework as a whole. A few specific examples are mentioned at the end of Section~\ref{sec:comm-upper}. One particularly glaring gap mentioned in Section~\ref{sec:list} is the absence of any low-degree upper bound that tracks \emph{super-polynomial} runtime. A specific open problem is to show that, in the planted submatrix model, $\MMSE_{\le D}$ is small for $D \approx \lambda^{-2}$ (see~\cite[Section~2.2]{sharp-est}).

\end{enumerate}

\newpage
\phantomsection
\addcontentsline{toc}{section}{References}

\bibliographystyle{alpha}
\bibliography{main}

\end{document}